\documentclass[reqno]{amsart}
\pdfoutput=1 

\usepackage[english]{babel}
\usepackage[utf8]{inputenc}
\usepackage[T1]{fontenc}
\usepackage{lmodern} \normalfont 
\DeclareFontShape{T1}{lmr}{bx}{sc} { <-> ssub * cmr/bx/sc }{}
\usepackage{anyfontsize}
\usepackage{microtype}	

\usepackage{amssymb}
\usepackage{amsmath}
\usepackage{amsthm}
\usepackage{mathtools}
\usepackage{mathrsfs}
\usepackage{accents} 
\usepackage{subdepth} 
\usepackage{etoolbox}
\usepackage{siunitx}
\usepackage{dsfont} 
\usepackage{graphicx}
\usepackage{color}
\usepackage[dvipsnames, table]{xcolor}
\usepackage{tikz}
\usetikzlibrary{calc,positioning,shapes}
\usetikzlibrary{patterns,decorations.pathmorphing,decorations.markings}
\usetikzlibrary{intersections, backgrounds}
\usepackage{pgfplots}
\pgfplotsset{compat=newest}
\usepackage[margin=10pt,font=small,labelfont=bf,labelsep=endash]{caption}
\usepackage{subcaption}

\usepackage{paralist}

\usepackage{booktabs}
\usepackage{paralist}
\usepackage{soul}

\numberwithin{equation}{section}

\usepackage[ruled, lined, linesnumbered, commentsnumbered, longend]{algorithm2e}

\usepackage{cite}
\usepackage{multirow} 
\usepackage{enumitem} 
\setlist[enumerate]{label=(\roman*)}
\usepackage{xspace}

\usepackage[colorlinks,linkcolor=teal,citecolor=teal,urlcolor=teal]{hyperref}
\usepackage[nameinlink,noabbrev]{cleveref}

\theoremstyle{plain}
\newtheorem{theorem}{Theorem}[section]
\newtheorem{proposition}[theorem]{Proposition}
\newtheorem{lemma}[theorem]{Lemma}
\Crefname{lemma}{Lemma}{Lemmata}
\newtheorem{corollary}[theorem]{Corollary}
\newtheorem{remark}[theorem]{Remark}
\newtheorem{definition}[theorem]{Definition}

\newtheorem{example}[theorem]{Example}

\newcommand{\N}{\ensuremath\mathbb{N}}

\newcommand{\R}{\ensuremath\mathbb{R}}

\newcommand{\ri}{\ensuremath\mathrm{i}}

\newcommand{\T}{\ensuremath\mathsf{T}}

\newcommand{\GL}[1]{\mathrm{GL}_{#1}} 

\newcommand{\Spd}[1]{\mathcal{S}^{#1}_{\succ}}
\newcommand{\Spsd}[1]{\mathcal{S}^{#1}_{\succcurlyeq}}
\newenvironment{smallbmatrix}{\left[\begin{smallmatrix}}{\end{smallmatrix}\right]}
\newcommand{\OrthogonalGroup}[1]{\mathrm{O}_{#1}}

\newcommand{\dtau}{\,\mathrm{d}\tau}

\newcommand{\domega}{\,\mathrm{d}\omega}

\DeclareMathOperator{\range}{im}
\newcommand{\image}{\range}
\newcommand{\im}{\image}

\DeclareMathOperator{\rank}{rank}

\DeclareMathOperator{\Skew}{skew}
\DeclareMathOperator{\Sym}{sym}

\DeclareMathOperator{\tr}{tr}

\DeclareMathOperator{\vectorize}{vec}
\DeclareMathOperator{\vech}{vech}

\newcommand{\calD}{\mathcal{D}}

\newcommand{\calH}{\mathcal{H}}

\newcommand{\calJ}{\mathcal{J}}
\newcommand{\calL}{\mathcal{L}}

\newcommand{\calV}{\mathcal{V}}
\newcommand{\calW}{\mathcal{W}}

\newcommand{\Ltwo}{\calL_2}
\newcommand{\Linf}{\calL_\infty}
\newcommand{\Htwo}{\calH_2}

\newcommand{\norm}[1]{\left\|#1\right\|}

\definecolor{mycolor1}{rgb}{0.00000,0.44700,0.74100}
\definecolor{mycolor2}{rgb}{0.85000,0.32500,0.09800}
\definecolor{mycolor3}{rgb}{0.92900,0.69400,0.12500}
\definecolor{mycolor4}{rgb}{0.46600,0.67400,0.18800}
\definecolor{mycolor5}{rgb}{0.49400,0.18400,0.55600}

\newcommand{\inv}[1]{{#1}^{-1}}

\newcommand{\system}{\Sigma}
\newcommand{\state}{x}
\newcommand{\stateDim}{n}
\newcommand{\reduce}[1]{\tilde{#1}}
\newcommand{\systemRed}{\reduce{\system}}
\newcommand{\stateRed}{\reduce{\state}}
\newcommand{\stateDimRed}{r}

\newcommand{\inpVar}{u}

\newcommand{\inpVarDim}{m}

\newcommand{\outVar}{y}
\newcommand{\outVarHam}{\outVar_{\hamiltonian}}
\newcommand{\outVarRed}{\reduce{\outVar}}
\newcommand{\outVarHamRed}{\reduce{\outVar}_{\hamiltonian}}

\newcommand{\outVarDim}{p}

\newcommand{\systemPH}{\system_{\mathsf{pH}}}
\newcommand{\systemPHRed}{\reduce{\system}_{\mathsf{pH}}}
\newcommand{\systemHam}{\system_{\hamiltonian}}
\newcommand{\systemHamRed}{\reduce{\system}_{\hamiltonian}}

\newcommand{\hamiltonian}{\calH}
\newcommand{\hamiltonianHess}{Q}

\newcommand{\hamiltonianHessRed}{\reduce{\hamiltonianHess}}
\newcommand{\structureMatrix}{\Gamma}
\newcommand{\dissipationMatrix}{W}

\newcommand{\Xmin}{X_{\min}}
\newcommand{\Xmax}{X_{\max}}
\newcommand{\KYPset}{\mathbb{X}}

\newcommand{\ctrlG}{\mathcal{P}}
\newcommand{\obsG}{\mathcal{O}}

\newcommand{\subQO}[1]{#1_{\mathsf{QO}}}
\newcommand{\systemQO}{\subQO{\system}}
\newcommand{\systemQOred}{\subQO{\reduce{\system}}}

\newcommand{\obsGQO}{\subQO{\obsG}}

\newcommand{\obsGQOred}{\subQO{\reduce{\obsG}}}
\newcommand{\sylY}{Y}
\newcommand{\sylZ}{Z}

\newcommand{\co}{\mathsf{c}}
\newcommand{\uc}{\overline{\co}}
\newcommand{\ob}{\mathsf{o}}
\newcommand{\uo}{\overline{\ob}}

\newcommand{\objectiveFunc}{\calJ}

\newcommand{\Proj}{\mathfrak{P}}

\newcommand{\acronym}[1]{\textsf{#1}\xspace}
\newcommand{\FOM}{\acronym{FOM}} 	
\newcommand{\FOMs}{\acronym{FOMs}} 	
\newcommand{\MOR}{\acronym{MOR}} 	
\newcommand{\ROM}{\acronym{ROM}} 	
\newcommand{\ROMs}{\acronym{ROMs}} 	
\newcommand{\pH}{\acronym{pH}} 		
\newcommand{\PH}{\acronym{\PH}} 	
\newcommand{\LTI}{\acronym{LTI}} 	
\newcommand{\LTIQO}{\acronym{LTIQO}} 
\newcommand{\QB}{\acronym{QB}} 		
\newcommand{\KYP}{\acronym{KYP}} 	
\newcommand{\IRKA}{\acronym{IRKA}} 	
\newcommand{\pHIRKA}{\acronym{pH-IRKA}} 
\newcommand{\PRBT}{\acronym{PRBT}} 	
\newcommand{\EMPRBT}{\acronym{EM-PRBT}} 
\newcommand{\AAA}{\acronym{AAA}} 	
\newcommand{\ARE}{\acronym{ARE}} 	
\newcommand{\AREs}{\acronym{AREs}}
\newcommand{\SDP}{\acronym{SDP}}    

\newcommand{\LMI}{\acronym{LMI}}    

\DeclareMathOperator{\argmin}{argmin}

\title{Energy matching in reduced passive and port-Hamiltonian systems}
\author[T.~Holicki \and J.~Nicodemus \and P.~Schwerdtner \and B.~Unger]{Tobias Holicki${}^\star$ \and Jonas Nicodemus${}^\dagger$ \and Paul Schwerdtner$^{\ddagger}$ \and Benjamin Unger${}^\dagger$}

\address{${}^{\star}$ Department of Mathematics, University of Stuttgart, Pfaffenwaldring 5a, 70569 Stuttgart, Germany}
\email{tobias.holicki@imng.uni-stuttgart.de}

\address{${}^{\dagger}$ Stuttgart Center for Simulation Science (SC SimTech), University of Stuttgart, Universit\"{a}tsstr.~32, 70569 Stuttgart, Germany}
\email{\{jonas.nicodemus,benjamin.unger\}@simtech.uni-stuttgart.de}

\address{${}^{\ddagger}$ Courant Institute of Mathematical Sciences, New York University, New York, NY 10012, United States}
\email{paul.schwerdtner@nyu.edu}
\date{\today}

\begin{document}

\begin{abstract}
    It is well known that any port-Hamiltonian (pH) system is passive, and conversely, any minimal and stable passive system has a pH representation. Nevertheless, this equivalence is only concerned with the input-output mapping but not with the Hamiltonian itself. Thus, we propose to view a pH system either as an enlarged dynamical system with the Hamiltonian as additional output or as two dynamical systems with the input-output and the Hamiltonian dynamic. Our first main result is a structure-preserving Kalman-like decomposition of the enlarged pH system that separates the controllable and zero-state observable parts. Moreover, for further approximations in the context of structure-preserving model-order reduction (MOR), we propose to search for a Hamiltonian in the reduced pH system that minimizes the $\mathcal{H}_2$-distance to the full-order Hamiltonian without altering the input-output dynamic, thus discussing a particular aspect of the corresponding multi-objective minimization problem corresponding to $\mathcal{H}_2$-optimal MOR for pH systems. We show that this optimization problem is uniquely solvable, can be recast as a standard semidefinite program, and present two numerical approaches for solving it. The results are illustrated with three academic examples.
\end{abstract}

\maketitle

%

\section{Introduction}\label{sec:intro}
The \emph{port-Hamiltonian} (\pH) modeling paradigm offers an intuitive energy-based formulation of dynamical systems across a wide variety of physical domains such as electrical systems~\cite{EstT00, GunF99a, GunF99b}, fluid-flow problems~\cite{AltS17}, or mechanical multi-body systems~\cite[Ex.\,12]{BeaMXZ18}. By design, \pH systems are automatically stable and passive and can be coupled across different scales and physical domains, which makes them valuable building blocks for large network models~\cite{MehU22}. 
The \pH framework is particularly appealing due to its inherent \emph{Hamiltonian} structure. The Hamiltonian function, which represents the system's total energy, provides a powerful tool for understanding and analyzing the system's behavior.
Since first-principle \emph{full-order models} (\FOMs) of complex systems or large system networks often have a high state-space dimension, \emph{model order reduction} (\MOR) is necessary in many cases to enable efficient numerical simulations or even real-time model-based control by computing a \emph{reduced-order model} (\ROM) that is used instead. 
The current state-of-the-art system-theoretic \MOR methods for \pH systems aim to preserve the input-output mapping of the system. However, in the context of \pH systems, not only the input-output mapping is of relevance but also the approximation of the Hamiltonian. For instance, preserving the Hamiltonian during \MOR is crucial for many applications, such as energy-aware control synthesis~\cite{CalDRSK22, RasBZSJF22}.

In this article, we thus aim for a \MOR method that approximates the input-output mapping and, simultaneously, the Hamiltonian.
To this end, we offer a new perspective on the \MOR problem for \emph{linear time-invariant} (\LTI) \pH systems. In particular, we argue that \pH systems should not be treated merely as a special case of standard \LTI systems during \MOR but instead propose to view \pH systems as two dynamical (respectively an extended) dynamical systems consisting of the classical input-output mapping and, additionally, a dynamical system with a quadratic output representing the evolution of the Hamiltonian.

Exploiting the \emph{Kalman-Yakubovich-Popov} (\KYP) inequality, see the forthcoming \Cref{subsec:pH}, we propose a novel post-processing step called \emph{energy matching} --- to be performed after any structure-preserving \MOR method~--- for the \ROM such that the approximation error of the Hamiltonian dynamic is minimized, without changing the system's input-output dynamic. In more detail, we exploit the non-uniqueness of the \pH formulation to replace the Hessian of the Hamiltonian with any other positive-definite solution of the \KYP inequality without altering the input-output mapping. This allows us to formulate an optimization problem that minimizes the Hamiltonian approximation error.

Our main contributions, centered around the novel definition of the extended \pH system in \Cref{sec:eph}, are the following:
\begin{enumerate}
    \item As a natural first step towards system theoretical \MOR for this class of dynamical systems, we derive a structure-preserving Kalman-like decomposition in \Cref{sec:minimality}, which can be used as an efficient pre-computation step before applying any \MOR methods.
    \item Finally, we provide the new \emph{energy matching} post-processing algorithm in \Cref{sec:energy-matching}, which can be applied after any classical structure preserving \MOR method to minimize the error of the Hamiltonian approximation.
\end{enumerate}
We demonstrate the efficiency of the method using three numerical examples in \Cref{sec:num-exp}.

\subsection{Literature review and state-of-the-art}
\label{subsec:literature}

\MOR for standard \LTI systems of the form
\begin{equation}
  \label{eq:LTI}
  \system\quad \left\{\quad\begin{aligned}
      \dot{\state}(t) &= A\state(t) + B\inpVar(t),\\
      \outVar(t) &= C\state(t) + D\inpVar(t),
  \end{aligned}\right.
\end{equation}
where $A \in \R^{n \times n}, B \in \R^{n \times m}, C \in \R^{p \times n}$, and $D \in \R^{p \times m}$, is well understood. There exist several well-established algorithms that compute \ROMs of the form
\begin{equation}
    \label{eq:LTI:rom}
    \reduce{\system}\quad \left\{\quad\begin{aligned}
        \dot{\reduce{\state}}(t) &= \reduce{A}\reduce{\state}(t) + \reduce{B}\inpVar(t),\\
        \reduce{\outVar}(t) &= \reduce{C}\reduce{\state}(t) + \reduce{D}\inpVar(t),
    \end{aligned}\right.
\end{equation}
with matrices $\reduce{A}\in\R^{\stateDimRed\times\stateDimRed}$, $\reduce{B}\in\R^{\stateDimRed\times\inpVarDim}$, $\reduce{C}\in\R^{\outVarDim\times\stateDimRed}$, and $\reduce{D}\in\R^{\outVarDim\times\inpVarDim}$ that approximate the \FOM with high fidelity. One standard input-output error measure is the $\calH_2$-error (cf.~\cite[Sec. 7.2]{Mac13} and the references therein)
\begin{equation}
    \| \Sigma - \reduce{\Sigma} \|_{\Htwo} \vcentcolon= \sqrt{\frac{1}{2\pi} \int_{-\infty}^{\infty} \norm{H(\ri \omega) - \reduce{H}(\ri \omega)}^2_{\mathrm{F}} \domega},
\end{equation}
that measures the deviation of the \ROM transfer function $\reduce{H}$ from the \FOM transfer function $H$. These transfer functions are defined as
\begin{equation*}
    H(s) \vcentcolon= C{(s I_{\stateDim} - A)}^{-1} B + D \qquad\text{and}\qquad
     \reduce{H}(s) \vcentcolon= \reduce{C}{(s I_{\stateDimRed} - \reduce{A})}^{-1} \reduce{B} + \reduce{D}.
\end{equation*}
Moreover, we have that $\|y-\reduce{y}\|_{\Linf} \le \|H-\reduce{H}\|_{\Htwo} \|u\|_{\Ltwo}$ (again  cf.~\cite[Sec. 7.2]{Mac13}), which ensures that a small $\Htwo$-error leads to a good approximation of the input-output map in the $\Linf$-norm.

A comprehensive review of the classical \MOR methods is beyond the scope of this paper, and we refer to~\cite{Ant05, AntBG20, BenCOW17} for an overview of this topic. We mention that many of these methods employ a projection framework, i.e., they compute a \ROMs based on a subspace projection onto the $\stateDimRed$-dimensional subspaces $\im(V)$ and $\im(W)$ of $\R^{\stateDim}$ encoded via the matrices $V,W \in \R^{\stateDim \times \stateDimRed}$ with $W^\T V = I$, i.e., the \ROM matrices are defined as $\reduce{A} = W^\T A V$, $\reduce{B} = W^\T B$, $\reduce{C} = CV$, and $\reduce{D} = D$.
However, many of the standard methods, such as the \emph{iterative rational Krylov algorithm} (\IRKA), balanced truncation, and proper orthogonal decomposition, have no guarantee to preserve the \pH structure. Instead, specialized methods for \pH systems can be employed. These \MOR methods can roughly be divided into two main categories. Methods in the first category aim at a good approximation of the state, which should also yield a good approximation of the input-output map. Popular examples rely on symplectic model reduction \cite{PenM16, AfkH17,AfkH19,BucBH19} and proper orthogonal decomposition with compatibility conditions \cite{EggKLMM18}. 
Naturally, if methods from this category approximate the state sufficiently well, then they also provide a good approximation of the Hamiltonian. However, from a system-theoretic perspective, these methods may approximate parts of the state that are irrelevant to the input-output mapping. In contrast, methods in the second category directly approximate the input-output mapping of the dynamical system. Prominent examples are the structure-preserving variant of \IRKA \cite{GugPBV12}, optimization algorithms that aim at minimizing the $\mathcal{H}_2$- or $\mathcal{H}_\infty$-norms, cf.~\cite{SatS18,MosL20,SchV20,MorNU23}, and balancing methods such as \cite{BreMS22,BorSF21}. Moreover, exploiting the equivalence between \pH systems and passive systems (see~\cite{CheGH23, ReiRV15} for a thorough investigation) enables passivity preserving methods such as \emph{positive-real balanced truncation} (\PRBT) \cite{DesP84, ReiS10} and spectral factorization \cite{BreU22}. However, for the recovery of the \pH system, it needs to be clarified how to choose the Hamiltonian, which is what motivated this paper in the first place. Since our main focus is on the approximation of the input-output mapping and the Hamiltonian but not on the state, we focus here on methods from the second category, albeit our energy matching post-processing can also be applied to methods from the first category. 

\MOR methods for linear systems are evaluated based on their approximation of the input-output mapping, which can be assessed using well-established error measures (such as the $\calH_2$ norm) based on the transfer function distances. The evaluation of the approximation quality of the Hamiltonian requires a more advanced error analysis that has only recently been established. When we add the Hamiltonian as an additional output, an \LTI \pH system becomes a \emph{linear time-invariant system with quadratic output} (\LTIQO). \MOR for such systems is considered, e.g., in~\cite{VanM10, VanVLM12}, in which single output \LTIQO systems are simplified to standard \LTI systems with multiple outputs such that either balancing or Krylov-based \MOR methods can be applied. In~\cite{PulN19}, \LTIQO systems are rewritten as quadratic-bilinear (\QB) systems that are subsequently reduced via balanced truncation.

Our approach for approximating the Hamiltonian is based on developments in~\cite{BenGP22}, in which the $\calH_2$ error measure is extended to \LTIQO systems. Moreover, in~\cite{BenGP22}, energy functionals and Gramians are introduced for \LTIQO systems such that balanced truncation can be applied directly. Finally, in~\cite{GosA19}, an iterative structure preserving \MOR algorithm is presented based on solving two Sylvester equations and in~\cite{GosG22} the \emph{Adaptive Antoulas-Anderson} (\AAA) algorithm is extended to \LTIQO to develop a data-driven modeling framework. However, to our knowledge, there are no structure-preserving variants of the mentioned methods for \LTIQO systems.

\subsection{Organization of the manuscript}
Our manuscript is organized as follows: first, we recall the basics of the \pH framework in \Cref{sec:preliminaries}. The view of \pH systems as extended dynamical systems, particularly the Hamiltonian dynamic, are presented in \Cref{sec:eph} and minimality of the extended system is analyzed in \Cref{sec:minimality}.
We then present our proposed \MOR post-processing method for optimizing the Hamiltonian of a \ROM to match the Hamiltonian of the \FOM in \Cref{sec:energy-matching}. Finally, the method's efficiency is demonstrated in three numerical examples in \Cref{sec:num-exp}.

\subsection{Notation and abbreviations}
We use the symbols $\N$, $\R$, $\R^n$, $\R^{n\times m}$, $\GL{n}$, $\Spd{n}$, $\Spsd{n}$, and $\OrthogonalGroup{\stateDim}$ to denote the positive integers, the real numbers, the set of column vectors with $n\in\N$ real entries, the set of $n\times m$ real matrices, the set of nonsingular matrices, the set of symmetric positive definite, the set of symmetric positive semi-definite matrices, and the orthogonal matrices, respectively. For a matrix $A\in\R^{n\times m}$, we use the symbols $A^\T$, $\Sym(A) = \tfrac{1}{2}(A+A^\T)$, and $\Skew(A)=\tfrac{1}{2}(A - A^\T)$, for the transpose, the symmetric part, and the skew-symmetric part, respectively.

\section{Preliminaries}\label{sec:preliminaries}
We first recall a few basic notions from \LTI systems and \pH systems, that we will later use for our developments in \Cref{sec:eph}.

\subsection{Controllability and Observability}

An \LTI system such as~\eqref{eq:LTI} is called controllable or observable if the corresponding controllability and observability matrices have full row and column rank, respectively, i.e.,
\begin{displaymath}
    \rank\begin{bmatrix}
        B & AB & \cdots & A^{\stateDim-1}B
    \end{bmatrix} = \stateDim \qquad\text{and}\qquad
    \rank
    \begin{bmatrix}
        C & A^\T C & \cdots & {\big(A^\T\big)}^{\stateDim-1}C
    \end{bmatrix} = \stateDim.
\end{displaymath}
The system~\eqref{eq:LTI} is called minimal if it is controllable and observable. Controllability and observability are closely related to the (infinite) Gramians
\begin{equation}\label{eq:gramian:int}
    \ctrlG \vcentcolon= \int_0^\infty \exp(A\tau)BB^\T\exp(A^\T \tau) \dtau \quad\text{and}\quad 
    \obsG \vcentcolon= \int_0^\infty \exp(A^\T \tau)CC^\T \exp(A\tau)\dtau,
\end{equation}
which exist if the dynamical system~\eqref{eq:LTI} is asymptotically stable, i.e., if all eigenvalues of~$A$ are in the open left-half plane. In this case, the Gramians can be computed as solutions of the Lyapunov equations
\begin{subequations}\label{eq:gramian:lyap}
    \begin{align}
        \label{eq:gramian:lyap:ctrl}
        A\ctrlG + \ctrlG A^\T + BB^\T &= 0, \\
        \label{eq:gramian:lyap:obs}
        A^\T \obsG + \obsG A + C^\T C &= 0,
    \end{align}
\end{subequations}
respectively, and we have that $\system$ is controllable if and only if $\rank(\ctrlG) = \stateDim$, and observable if and only if $\rank(\obsG) = \stateDim$.

\subsection{Port-Hamiltonian systems and the Kalman-Yakubovich-Popov inequality}\label{subsec:pH}
We consider \LTI \pH systems defined as follows.

\begin{definition}[Port-Hamiltonian system~\cite{SchJ14}]\label{def:pHsys}
    An \LTI system of the form
    \begin{subequations}\label{eq:pH}
        \begin{equation}
        \label{eq:pH:sys}
        \systemPH\quad \left\{\quad\begin{aligned}
            \dot{\state}(t) &= (J-R)\hamiltonianHess\state(t) + (G-P)\inpVar(t),\\
            \outVar(t) &= {(G+P)}^\T\hamiltonianHess \state(t) + (S-N)\inpVar(t),
        \end{aligned}\right.
        \end{equation}
        with matrices $J,R,\hamiltonianHess\in\R^{\stateDim\times\stateDim}$, $G,P\in\R^{\stateDim\times\inpVarDim}$, $S,N\in\R^{\inpVarDim\times\inpVarDim}$, 
        together with a \emph{Hamiltonian} function
        \begin{equation}
        \label{eq:pH:Hamiltonian}
        \hamiltonian\colon \R^{\stateDim}\to\R,\qquad \state \mapsto \tfrac{1}{2} \state^\T \hamiltonianHess \state,
        \end{equation}
    \end{subequations}
    is called a \emph{port-Hamiltonian} system, if
    \begin{enumerate}
        \item the structure matrix
        $\structureMatrix \vcentcolon= 
        \begin{bsmallmatrix}
            J & G\\
            -\smash{G^\T} & N
        \end{bsmallmatrix}
        $
        is skew-symmetric,
        \item the dissipation matrix
        $
        \dissipationMatrix \vcentcolon= 
        \begin{bsmallmatrix}
            R & P\\
            P^\T & S
        \end{bsmallmatrix}
        $
        is symmetric positive semi-definite, and 
        \item the Hessian of the Hamiltonian $\hamiltonianHess$ is symmetric positive semi-definite.
    \end{enumerate}
    The variables $\state$, $\inpVar$, and $\outVar$ are referred to as the \emph{state}, \emph{input}, and \emph{output}, respectively.
\end{definition}

For such systems, structure-preserving \MOR computes \pH \ROMs
\begin{equation}
    \label{eq:pH:rom}
    \systemPHRed\quad \left\{\quad\begin{aligned}
        \dot{\reduce{\state}}(t) &= (\reduce{J}-\reduce{R})\reduce{\hamiltonianHess}\reduce{\state}(t) + (\reduce{G}-\reduce{P})\inpVar(t),\\
        \outVarRed(t) &= {(\reduce{G}+\reduce{P})}^\T\reduce{\hamiltonianHess} \reduce{\state}(t) + (\reduce{S}-\reduce{N})\inpVar(t),
    \end{aligned}\right.
\end{equation}
with matrices $\reduce{J},\reduce{R},\reduce{\hamiltonianHess}\in\R^{\stateDimRed\times\stateDimRed}$, $\reduce{G},\reduce{P}\in\R^{\stateDimRed\times\inpVarDim}$, $\reduce{S},\reduce{N}\in\R^{\inpVarDim\times\inpVarDim}$, that satisfy the same constraints as in \Cref{def:pHsys} but with $\stateDimRed \ll \stateDim$. Typically, \MOR (and also structure-preserving \MOR) aims to compute \ROMs such that $\outVar - \outVarRed$ is small for all admissible inputs $\inpVar$ in an appropriate norm (which results in a good approximation of the input-output mapping). The approximation of the Hamiltonian, i.e., $\hamiltonian - \reduce{\hamiltonian}$ in some appropriate norm is typically not considered; here $\reduce{\hamiltonian}$ denotes the Hamiltonian of the reduced system~\eqref{eq:pH:rom}, given by
\begin{equation*}
    \reduce{\hamiltonian}\colon \R^{\stateDimRed}  \to \R, \qquad \stateRed \mapsto \frac{1}{2} \stateRed^\T \reduce{\hamiltonianHess} \stateRed.    
\end{equation*}
We say that a general \LTI system as in~\eqref{eq:LTI} has a \pH representation whenever we can factorize the system matrices in the form of~\eqref{eq:pH:sys} with the properties given in \Cref{def:pHsys}. While the specific matrices of a \pH system are typically obtained during the modeling process, the factorization of the system matrices is generally not unique. Indeed, it is easily seen that a \pH system is passive, and vice versa, any stable and minimal passive system has a \pH representation; see for instance~\cite{BeaMX22}. If $\system$ in~\eqref{eq:LTI} is passive, then a \pH representation can be obtained via a symmetric positive-definite solution $X\in\Spd{\stateDim}$ of the \KYP inequality
\begin{equation}
    \label{eq:KYP}
    \calW_\system(X) \in \Spsd{\stateDim+\inpVarDim}
\end{equation}
with
\begin{displaymath}
    \calW_\system\colon \R^{\stateDim\times\stateDim}\to\R^{(\stateDim+\inpVarDim)\times(\stateDim+\inpVarDim)},\qquad X \mapsto 
    \begin{bmatrix}
        -A^\T X - XA & C^\T - XB\\
        C-B^\T X & D+D^\T
    \end{bmatrix}.
\end{displaymath}
In more detail, defining the set $\KYPset_{\system} \vcentcolon= \{X\in\Spd{\stateDim} \mid \calW_{\system}(X)\in\Spsd{\stateDim+\inpVarDim}\}$, 
it is easy to verify that for a passive \LTI system~\eqref{eq:LTI}, any $X\in\KYPset_{\Sigma}$ of~\eqref{eq:KYP} yields a \pH representation by setting
\begin{gather}
    \label{eqs:pH:decomposition}
    \hamiltonianHess \vcentcolon= X, \quad J \vcentcolon= \Skew(AX^{-1}), \quad R \vcentcolon= -\Sym(AX^{-1}) \\
    G \vcentcolon= \tfrac{1}{2}(X^{-1}C^\T+B), \quad P \vcentcolon= \frac{1}{2}(X^{-1}C^\T - B), \quad S \vcentcolon= \Sym(D), \quad N \vcentcolon= \Skew(D).
\end{gather}
Note that we have
\begin{equation}
    \label{eq:pH:decomposition:check}
    (J-R)\hamiltonianHess = \tfrac{1}{2}\left(AX^{-1} - X^{-1}A^\T + AX^{-1} + X^{-1}A^\T\right)X = A,
\end{equation}
and similarly for the other matrices. Hence, the \pH representation does not affect the state-space description~\eqref{eq:LTI}, but is merely a special decomposition of the system matrices.
For our forthcoming analysis, we gather several results from the literature~\cite{LanT85, CamIV14, Wil72} about the \KYP inequality~\eqref{eq:KYP}.
\begin{theorem}
    \label{thm:KYP}
    Consider the dynamical system~$\Sigma$ in~\eqref{eq:LTI} and the associated \KYP inequality~\eqref{eq:KYP}.
    \begin{enumerate}
        \item\label{thm:KYP:symPosSemi} If the dynamical system is asymptotically stable, i.e., the eigenvalues of $A$ are in the open left half plane, then any solution $X\in\R^{\stateDim\times\stateDim}$ of~\eqref{eq:KYP} is symmetric positive semi-definite.
        \item\label{thm:KYP:posDef} If the dynamical system is observable, then any solution $X\in\Spsd{\stateDim}$ of~\eqref{eq:KYP} is positive definite.
        \item\label{thm:KYP:bounded} Suppose the dynamical system is minimal and asymptotically stable. Then there exist matrices $\Xmin,\Xmax\in\KYPset_{\system}$ such that any~$X\in\KYPset_{\system}$ satisfies
        \begin{displaymath}
            \Xmin \preccurlyeq X \preccurlyeq \Xmax.
        \end{displaymath}
        In particular, the set $\KYPset_{\Sigma}$ is bounded.
    \end{enumerate}
\end{theorem}
\begin{proof}
    Since the results are well-known, we simply refer to the respective literature.
    \begin{enumerate}
        \item Let $X\in\R^{\stateDim\times\stateDim}$ be a solution of~\eqref{eq:KYP}. Then there exists a matrix $M\in\Spsd{\stateDim}$ such that 
        \begin{displaymath}
            -A^\T X - XA^\T = M.
        \end{displaymath}					
        The result is thus an immediate consequence of~\cite[Cha.~12.3, Thm.~3]{LanT85}.
        \item See~\cite[Prop.~1]{CamIV14}.
        \item See~\cite[Thm.~3]{Wil72}.\qedhere
    \end{enumerate}
\end{proof}

If $D$ is regular, solutions of the \KYP that minimize $\rank(\calW_\system(\cdot))$ can be computed by solving an associated \emph{algebraic Riccati equation} (\ARE) of the form
\begin{align}
    \label{eq:passivity-riccati}
    A^\T X +XA + (-C^\T + XB){(D+D^\T)}^{-1}(-C+B^\T X) = 0.
\end{align}
The connection between solutions of this \ARE and the \KYP are studied in great detail in~\cite{Wil71}. Numerical solvers for the \ARE are readily available
and can be used to compute both \emph{minimal} and \emph{maximal solutions}, which are also the minimal and maximal solutions of the \KYP inequality from \Cref{thm:KYP}\,\ref{thm:KYP:bounded}. These solutions have the property that for each solution $X$ of the \ARE, we have that $X-X_{\min} \in \Spsd{n}$ and $X_{\max} - X \in \Spsd{n}$. Moreover, each solution of the \ARE can be constructed as $X = X_{\max}\Proj + X_{\min}(I-\Proj)$, where $\Proj$ and $I-\Proj$ are projections onto invariant subspaces of associated matrices; see~\cite{Wil71} for further details. 


\section{Extended port-Hamiltonian systems}\label{sec:eph}
As already motivated, our goal is to find a surrogate that well approximates the input-output behavior of a given \pH system \eqref{eq:pH:sys} and its Hamiltonian \eqref{eq:pH:Hamiltonian} simultaneously. To this end, it is instrumental to introduce the corresponding extended system  
\begin{equation}
	\label{eq:ph:ext}
	\Sigma_{\mathsf{epH}}\quad\left\{\quad\begin{aligned}
		\dot{\state}(t) &= (J-R)\hamiltonianHess\state(t) + (G-P)\inpVar(t),\\
		\outVar(t) &= (G+P)^\T \hamiltonianHess\state(t) + (S-N)\inpVar(t),\\
		\outVarHam(t) &= \tfrac{1}{2}\state(t)^\T \hamiltonianHess \state(t).
	\end{aligned}\right.
\end{equation}
To assess the quality of the surrogate, we rely on the distance in the $\Htwo$-norm for linear systems with multiple linear and quadratic outputs. 
This norm is introduced in~\cite{BenGP22} and can be written as
\begin{equation}
    \label{eq:extendedNorm}
	\|\cdot\|_{\mathcal{H}_2}\colon \Sigma_{\mathsf{epH}} \mapsto \sqrt{\|\systemPH\|_{\Htwo}^2 + \|\systemHam\|_{\Htwo}^2}.
\end{equation}
Here, $\Sigma_{\mathsf{pH}}$ stands for the system \eqref{eq:pH:sys}, i.e., the system corresponding to the linear output of \eqref{eq:ph:ext}, and 
$\Sigma_{\mathcal{H}}$ denotes the system
\begin{equation}
	\label{eq:HamiltonianDynamic}
	\Sigma_{\mathcal{H}}\quad \left\{\quad
	\begin{aligned}
		\dot{\state}(t) &= (J-R)\hamiltonianHess\state(t) + (G-P)\inpVar(t),\\
		\outVarHam(t) &= \tfrac{1}{2}\state(t)^\T \hamiltonianHess \state(t),
	\end{aligned}\right.
\end{equation}
which is a linear system with a single quadratic output.
We refer to $\systemHam$ as the \emph{Hamiltonian dynamic} associated with the \pH system~\eqref{eq:pH} or with the extended system \eqref{eq:ph:ext}. 
If we abbreviate $A \vcentcolon= (J-R)\hamiltonianHess$ and $B\vcentcolon= G-P$, then (cf.~\cite{BenGP22,PrzPGB24})
\begin{equation*}
	\|\systemHam\|^2_{\Htwo} \vcentcolon = \tr(B^\T \obsGQO B) = \tfrac{1}{4}\tr(\ctrlG \hamiltonianHess \ctrlG \hamiltonianHess)
\end{equation*}
where $\obsGQO$ denotes the unique solution of the Lyapunov equation 
\begin{equation}
    \label{eq:LTIQO:obs-gramian}
		A^\T \obsGQO + \obsGQO A + \tfrac{1}{4}Q \ctrlG Q = 0
\end{equation}
and $\ctrlG$ being the controllability Gramian given by the solution of~\eqref{eq:gramian:lyap:ctrl}. We refer to \cite{BenGP22, GosA19, PrzPGB24} for detailed discussions on linear systems with quadratic outputs and the corresponding norm $\|\cdot\|_{\Htwo}$. At this point, we recall two useful properties of the latter norm that we will employ later. First, we have
\begin{equation*}
	\|\outVarHam\|_{\Linf} \leq \|\systemHam\|_{\Htwo}
    \|\inpVar\|_{\Ltwo}^2.
	%
\end{equation*}
Second, let
\begin{align*}
	\systemQOred\quad\left\{\quad\begin{aligned}
		\dot{\stateRed}(t) &= \reduce{A}\stateRed(t) + \reduce{B}\inpVar(t),\\
		\outVarRed(t) &= \tfrac{1}{2}\stateRed(t)^\T \reduce{Q} \stateRed(t)
	\end{aligned}\right.
\end{align*}
be another \LTIQO system. Then the squared distance between $\system_{\mathcal{H}}$ and $\systemQOred$ is given by 
\begin{equation*}
	\begin{aligned}
		\|\system_{\mathcal{H}} - \systemQOred\|_{\Htwo}^2 &= \tr(B^\T \obsGQO B) + \tr(\reduce{B}^\T\obsGQOred\reduce{B}) - 2\tr(B^\T \sylZ\reduce{B})\\
		&= \tfrac{1}{4} \tr(\ctrlG \hamiltonianHess \ctrlG \hamiltonianHess) + \tfrac{1}{4} \tr(\reduce{\ctrlG} \reduce{\hamiltonianHess} \reduce{\ctrlG} \reduce{\hamiltonianHess}) - \tfrac{1}{2} \tr(Y^\T \hamiltonianHess Y \reduce{\hamiltonianHess}).
	\end{aligned}
\end{equation*}
Here, $\obsGQOred$ is defined analogously to $\obsGQO$ and the rectangular matrices  $Z$ and $Y$ are the unique solutions of the Sylvester equations
\begin{align}
	\label{eq:h2inner:sylv:o}
	A^\T Z + Z \reduce{A} + \tfrac{1}{4}Q Y \reduce{Q} &= 0,\\
	\label{eq:h2inner:sylv:c}
    A Y + Y \reduce{A}^\T + B \reduce{B}^\T &= 0.
\end{align}
In the following, we will use the second formulation of the $\Htwo$-error for \LTIQO systems as cost functional for optimizing $\hamiltonianHessRed$.
This formulation is computationally more efficient since the matrices $\ctrlG, \reduce{\ctrlG}$, and $\sylY$ are independent of $\hamiltonianHessRed$ and can be precomputed. In contrast, the first formulation requires the solution $Z$ of the Sylvester equation~\eqref{eq:h2inner:sylv:o} and the computation of the quadratic observability Gramian $\obsGQOred$ for each $\hamiltonianHessRed$.

\section{Minimality of extended pH systems}\label{sec:minimality}

A reduction technique that does not involve any approximation error for the input-output behavior of a given standard \LTI system is the computation of a minimal realization, for example, based on the Kalman decomposition. Such techniques are beneficial as a preprocessing step for numerical methods; we refer to the numerical examples for further details. However, the Kalman decomposition generally does not preserve the structure of the given system. 
Since we are dealing with \pH systems, we are particularly interested in evaluating the Hamiltonian along system trajectories next to the system's input-output behavior. Thus, we develop a Kalman-like decomposition that permits the construction of a \ROM that preserves the input-output behavior of the extended system \eqref{eq:ph:ext}.
Let us begin with an example to demonstrate that it is not sufficient to rely on a minimal realization for the standard \pH system~\eqref{eq:pH:sys} while ignoring the Hamiltonian.
\begin{example}\label{ex:HamMinIoNot}
    Let $\stateDim=2$, $\hamiltonianHess = I_2$, $S=N=0$, $P=0$, and
    \begin{align}
        \label{eq:exPHmini:FOM}
        J &= 
        \begin{bmatrix}
            0 & -1\\
            1 & 0
        \end{bmatrix}, &
        R &= 
        \begin{bmatrix}
            1 & -1\\
            -1 & 2
        \end{bmatrix}, & 
        A &= (J-R)\hamiltonianHess = 
        \begin{bmatrix}
            -1 & 0\\
            2 & -2
        \end{bmatrix}, &
        G &= 
        \begin{bmatrix}
            1\\0
        \end{bmatrix}.
    \end{align}
    It is easy to see that with this choice, the input-output dynamic $\systemPH$ is controllable but not observable. A minimal realization is given by the dynamical system~\eqref{eq:LTI} with 
    \begin{align}
        \label{eq:exPHmini:ROM}
        \stateDimRed&=1, &
        \reduce{A}&=-1, &
        \reduce{B}&=1, & 
        \reduce{C}&=1.
    \end{align}
    The minimal realization is passive with unique solution $\hamiltonianHessRed^\star = 1$ of the \KYP inequality~\eqref{eq:KYP}. Nevertheless, straightforward computations\footnote{Take for instance a nonzero, constant control input and explicitly compute the Hamiltonian dynamic~$\systemHam$.} show that the Hamiltonian dynamic of the original system and the one of the minimal realization do not coincide, which is also reflected in the $\Htwo$-error $\| \systemHam - \systemHamRed\|_{\Htwo} = \tfrac{1}{6}$. We conclude that while~\eqref{eq:exPHmini:ROM} constitutes a minimal realization for the input-output dynamic of~\eqref{eq:exPHmini:FOM}, the reduced system given by~\eqref{eq:exPHmini:ROM} introduces an approximation error for the Hamiltonian dynamic. 
\end{example}

Towards a structure-preserving Kalman-like decomposition, let $V\in\OrthogonalGroup{\stateDim}$ (i.e., $V^\T V = I_{\stateDim}$) such that
\begin{displaymath}
    V^\T \hamiltonianHess V = 
    \begin{bmatrix}
        \hamiltonianHess_{\ob} & 0\\
        0 & 0
    \end{bmatrix}
\end{displaymath}
with $\hamiltonianHess_{\ob}\in\Spd{\rank(\hamiltonianHess)}$. Setting
\begin{gather*}
    V^\T \state = 
    \begin{bmatrix}
        \state_{\ob}\\
        \state_{\uo}
    \end{bmatrix}, \qquad
    V^\T (J-R) V = 
    \begin{bmatrix}
        J_{\ob} - R_{\ob} & -J_{\uo}^\T - R_{\uo}^\T\\
        J_{\uo} - R_{\uo} & J_{\star}-R_{\star} 
    \end{bmatrix}, \qquad
    V^\T G = 
    \begin{bmatrix}
        G_{\ob}\\
        G_{\uo}
    \end{bmatrix}, \qquad
    V^\T P = 
    \begin{bmatrix}
        P_{\ob}\\
        P_{\uo}
    \end{bmatrix}
\end{gather*}
we immediately observe that the dynamic corresponding to the state~$\state_{\uo}$ is not observable and hence can be removed without altering the input-output mapping. In particular, the \pH system
\begin{equation}
  \label{eq:ph:ext:nonsingHam}
  \Sigma_{\mathsf{epH},\ob}\quad \left\{\quad\begin{aligned}
      \dot{\state}_{\ob}(t)
                 &= \left(J_{\ob}-R_{\ob}\right)\hamiltonianHess_{\ob} 
                 \state_{\ob}(t) + 
                 (G_{\ob}-P_{\ob})\inpVar(t)\\
      \outVar(t) &= (G_{\ob} + P_{\ob})^\T \hamiltonianHess_{\ob}\state_{\ob}(t) + (S-N)\inpVar(t)\\
      \outVarHam(t) &= \tfrac{1}{2} \state_{\ob}(t)^\T \hamiltonianHess_{\ob} \state_{\ob}(t)
  \end{aligned}\right.
\end{equation}
has the same input-output mapping as~\eqref{eq:ph:ext}, i.e., $\|\Sigma_{\mathsf{epH}} - \Sigma_{\mathsf{epH},\ob}\|_{\Htwo} = 0$. In particular,  from an approximation perspective, we can always assume $\hamiltonianHess\in\Spd{\stateDim}$ since we can remove $\state_{\uo}$ without altering the output of the extended system, which is considered favorable from a modeling perspective \cite[Sec.~4.3]{MehU22}.

\begin{theorem}
  	\label{thm:KalmanContrFormPHQO}
  	Consider the \pH system~\eqref{eq:ph:ext} with $\hamiltonianHess\in\Spd{\stateDim}$. Then, there exists a matrix $V\in\GL{\stateDim}$ such that a state-space transformation with $V$ transforms the \pH system~\eqref{eq:ph:ext} into 
  	\begin{equation}
    	\label{eq:ph:ext:KalmanContrForm}
    	\left\{\quad
		\begin{aligned}
			\begin{bmatrix}
				\dot{\state}_{\co}(t)\\
				\dot{\state}_{\uc}(t)
			\end{bmatrix} &= \begin{bmatrix}
				\left(J_{\co} - R_{\co}\right) & 0\\
				0 & \left(J_{\uc} - R_{\uc}\right)
			\end{bmatrix}\begin{bmatrix}
				\state_{\co}(t)\\
				\state_{\uc}(t)
			\end{bmatrix} + \begin{bmatrix}
				G_{\co}-P_{\co}\\
				0\\
			\end{bmatrix}\inpVar(t),\\
        	\outVar(t) &= (G_{\co} + P_{\co})^\T \state_{\co}(t) + (S-N)\inpVar(t),\\
        	\outVarHam(t) &= \tfrac{1}{2} \state_{\co}(t)^\T \state_{\co}(t) + \tfrac{1}{2}\state_{\uc}(t)^\T \state_{\uc}(t)
    	\end{aligned}\right.
  	\end{equation}
  	with
  	\begin{align*}
    	V\begin{bmatrix}\state_{\co}\\\state_{\uc}\end{bmatrix} &= \state, &
    	\begin{bmatrix}
      		J_{\co} - R_{\co} & 0\\
      		0 & J_{\uc} - R_{\uc}
    	\end{bmatrix} &= V^{-1} (J-R)\hamiltonianHess V, &
    	\begin{bmatrix}
      		G_{\co} - P_{\co}\\
      		0
    	\end{bmatrix} &= V^{-1} (G-P)
  	\end{align*}
  	such that the subsystem corresponding to $\state_{\co}$ is in \pH-form and controllable.
\end{theorem}

\begin{proof}
  Let $\hamiltonianHess = LL^\T$ denote the Cholesky decomposition of the Hessian of the Hamiltonian, and define
  \begin{align*}
    \tilde{J} &\vcentcolon= L^\T J L, & 
    \tilde{R} &\vcentcolon= L^\T R L, & 
    \tilde{G} &\vcentcolon= L^\T G, &
    \tilde{P} &\vcentcolon= L^\T P.
  \end{align*}
  Using a classical Kalman decomposition, let $\tilde{V}\in\OrthogonalGroup{\stateDim}$ be such that
  \begin{displaymath}
    \left(\tilde{V}^\T (\tilde{J}-\tilde{R})\tilde{V}, \tilde{V}^\T (\tilde{G}-\tilde{P})\right) = \left(\begin{bmatrix}
        J_{\co} -R_{\co} & J_\star - R_\star\\
        0 & J_{\uc} - R_{\uc}
        \end{bmatrix}, \begin{bmatrix}
        G_{\co} - P_{\co}\\
        0
    \end{bmatrix}\right)
  \end{displaymath}
  is such that $(J_{\co}-R_{\co},G_{\co}-P_{\co})$ is controllable. Note that the transformation is a congruence transformation, such that we conclude $J_\star = 0 = R_\star$. The result follows with $V \vcentcolon= L^{-\T}\tilde{V}$.
\end{proof}

\begin{corollary}
	\label{cor:pHQO:KalmanContrForm}
	Consider the system~\eqref{eq:ph:ext:KalmanContrForm} with $J_{\co}\in\R^{n_{\co}\times n_{\co}}$ and $J_{\uc}\in\R^{n_{\uc}\times n_{\uc}}$. Then~\eqref{eq:ph:ext:KalmanContrForm} is zero-state observable. It is controllable, if and only if $n_{\uc} = 0$. In this case, asymptotic stability implies that the controllability and observability Gramians defined in~\eqref{eq:gramian:int} and \eqref{eq:LTIQO:obs-gramian} are positive definite.
\end{corollary}

\begin{proof}
    For zero-state observability, observe that $\inpVar\equiv 0$ implies $\state_{\co}(t) = \exp((J_{\co}-R_{\co})t)\state_{\co,0}$ and $\state_{\uc}(t) = \exp((J_{\uc}-R_{\uc})t)\state_{\uc,0}$. In particular, using
    \begin{displaymath}
        \outVarHam(t) = \tfrac{1}{2}\|\state_{\co}(t)\|_2^2 + \tfrac{1}{2}\|\state_{\uc}(t)\|_2^2
    \end{displaymath}
    yields $\outVarHam \equiv 0$ if and only if $\state_{\co,0} = 0$ and $\state_{\uc,0} = 0$. Controllability is a consequence of \Cref{thm:KalmanContrFormPHQO}, which also implies positive definiteness of the controllability Gramian. The positive definiteness of the observability Gramian follows immediately from \cite[Cha.~12.3, Thm.~3]{LanT85}.
\end{proof}

\begin{example}
	Since the proof of \Cref{thm:KalmanContrFormPHQO} is constructive, we can directly apply it to the system in \Cref{ex:HamMinIoNot}. Clearly, $Q = I_2\in\Spd{2}$ and the tuple $(A,B)$ is controllable. Hence, the system is already in the form of \Cref{thm:KalmanContrFormPHQO} with no uncontrollable states. Consequently, \Cref{cor:pHQO:KalmanContrForm} implies that the system is minimal. 
\end{example}

Summarizing the previous discussion, we obtain the main result of this section, namely a Kalman-like decomposition for \pH systems.

\begin{theorem}
  \label{thm:KalmanFormPHQO}
  Consider the \pH system~\eqref{eq:ph:ext}. Then, there exists a matrix $V\in\GL{\stateDim}$ such that a state-space transformation with $V$ transforms the \pH system~\eqref{eq:ph:ext} into 
  \begin{equation}
    \label{eq:ph:ext:KalmanForm}
    \left\{\quad\begin{aligned}
    	\begin{bmatrix}
    		\dot{\state}_{\co\ob}(t)\\
    		\dot{\state}_{\uc\ob}(t)\\
    		\dot{\state}_{\co\uo}(t)\\
    		\dot{\state}_{\uc\uo}(t)
    	\end{bmatrix} &= \begin{bmatrix}
    		J_{\co\ob} - R_{\co\ob} & 0 & 0 & 0\\
    		0 & J_{\uc\ob} - R_{\uc\ob} & 0 & 0\\
    		J_{\co\uo,1} - R_{\co\uo,1} & J_{\co\uo,2} - R_{\co\uo,2} & 0 & 0\\
    		0 & 0 & 0 & 0
    	\end{bmatrix}\begin{bmatrix}
    		\state_{\co\ob}(t)\\
    		\state_{\uc\ob}(t)\\
    		\state_{\co\uo}(t)\\
    		\state_{\uc\uo}(t)
    	\end{bmatrix} + \begin{bmatrix}
    		G_{\co\ob} - P_{\co\ob}\\
    		0\\
    		G_{\co\uo} - P_{\co\uo}\\
    		0
    	\end{bmatrix}\inpVar(t),\\
        \outVar(t) &= (G_{\co\ob} + P_{\co\ob})^\T \state_{\co\ob}(t) + (S-N)\inpVar(t),\\
        \outVarHam(t) &= \tfrac{1}{2} \state_{\co\ob}^\T(t) \state_{\co\ob}(t) + \tfrac{1}{2}\state_{\uc\ob}^\T(t) \state_{\uc\ob}(t)
    \end{aligned}\right.
  \end{equation}
  such that 
  \begin{enumerate}
  	\item the subsystem corresponding to~$\state_{\co\ob}$ is in \pH form, controllable, and zero-state observable,
  	\item the subsystem corresponding to~$\state_{\uc\ob}$ is in \pH form and zero-state observable,
  	\item the subsystem corresponding to~$\state_{\co\ob}$ and $\state_{\uc\ob}$ is zero-state observable, and
  	\item the subsystem corresponding to~$\state_{\co\ob}$ and $\state_{\co\uo}$ is controllable.
  \end{enumerate}
\end{theorem}

\begin{proof}
	Using a classical Kalman decomposition, let $\calV_{\co}\subseteq\R^{\stateDim}$ denote the controllability space associated with~\eqref{eq:ph:ext} and define the spaces
	\begin{align*}
		\calV_{\uc\uo} &\vcentcolon= \calV_{\co}^\perp \cap \ker(\hamiltonianHess), &
		\tilde{\calV}_{\co\uo} &\vcentcolon= \calV_{\co} \cap \ker(\hamiltonianHess), &
		\calV_1 &\vcentcolon= (\calV_{\uc\uo} + \tilde{\calV}_{\co\uo})^\perp.
	\end{align*}
	Using $\calV_{\co}+\calV_{\co}^\perp = \R^n$, we conclude $\ker(\hamiltonianHess) \perp \calV_1$. Assume that the columns of $V_1, V_{\co\uo},V_{\uc\uo}$ form a basis for $\calV_1,\calV_{\co\uo},\calV_{\uc\uo}$ such that $\tilde{V} = [V_1,V_{\co\uo},V_{\uc\uo}] \in\OrthogonalGroup{\stateDim}$. Define $\hamiltonianHess_1 \vcentcolon= V_1^\T \hamiltonianHess V_1$, $J_1 \vcentcolon= V_1^\T JV_1$, and $R_1 \vcentcolon= V_1^\T RV_1$. A state-space transformation of~\eqref{eq:ph:ext} with~$\tilde{V}$ then yields
	\begin{align*}
		\left\{\quad\begin{aligned}
		\begin{bmatrix}
			\dot{\state}_1(t)\\
			\dot{\state}_{\co\uo}(t)\\
			\dot{\state}_{\uc\uo}(t)
		\end{bmatrix} &= \begin{bmatrix}
			J_1-R_1 & 0 & 0\\
			J_{\co\uo} - R_{\co\uo} & 0 & 0\\
			0 & 0 & 0
		\end{bmatrix}\begin{bmatrix}
			\hamiltonianHess_1\state_1(t)\\
			\state_{\co\uo}(t)\\
			\state_{\uc\uo}(t)
		\end{bmatrix} + \begin{bmatrix}
			G_1 - P_1\\
			G_{\co\uo} - P_{\co\uo}\\
			0
		\end{bmatrix}\inpVar(t),\\
		\outVar(t) &= (G_1^\T + P_1^\T)\hamiltonianHess_1\state_1(t) + (S-N)\inpVar(t),\\
		\outVarHam(t) &= \tfrac{1}{2} \state_1^\T(t) \hamiltonianHess_1\state_1(t),
		\end{aligned}\right.
	\end{align*}
	where the subsystem corresponding to $\state_1$ is in \pH form.
	The result follows from applying \Cref{thm:KalmanContrFormPHQO} and \Cref{cor:pHQO:KalmanContrForm} to the \pH subsystem corresponding to $\state_1$.
\end{proof}

\begin{corollary}
    \label{cor:minreal}
	Consider the \pH system~\eqref{eq:ph:ext} with initial value $\state(0) = 0$ and, using the notation of \Cref{thm:KalmanFormPHQO}, the reduced controllable and zero-state observable \pH system
	\begin{equation}
		\Sigma_{\mathsf{epH},\co\ob}\quad \left\{\quad\begin{aligned}
    		\dot{\state}_{\co\ob}(t) &= (J_{\co\ob} - R_{\co\ob})\state_{\co\ob}(t) + (G_{\co\ob} - P_{\co\ob})\inpVar(t),\\
        \outVarRed(t) &= (G_{\co\ob} + P_{\co\ob})^\T \state_{\co\ob}(t) + (S-N)\inpVar(t),\\
        \outVarHamRed(t) &= \tfrac{1}{2} \state_{\co\ob}^\T(t) \state_{\co\ob}(t)
		\end{aligned}\right.
	\end{equation}
	with initial value $\state_{\co\ob}(0) = 0$. Then $\|\Sigma_{\mathsf{epH}} - \Sigma_{\mathsf{epH},\co\ob}\|_{\Htwo} = 0$. In particular, $\outVar \equiv \outVarRed$ and $\outVarHam \equiv \outVarHamRed$ for any control input $\inpVar$.
\end{corollary}

\begin{remark}
	A Kalman decomposition for \pH systems considering only the input-output dynamic~\eqref{eq:pH:sys} is obtained in \cite{PolS08}, which however requires certain invertibility assumptions to preserve the \pH structure and, moreover, does not take the Hamiltonian dynamic~\eqref{eq:HamiltonianDynamic} into consideration. Conversely, if an unstructured potentially non-minimal \LTI system of the form~\eqref{eq:LTI} is available (without any further information of a Hamiltonian), then the methods presented in \cite{BeaMX22} can be used to determine if the system can be recast as a \pH system.
\end{remark}


\section{Energy matching algorithm for surrogate models}
\label{sec:energy-matching}
If the Kalman-like decomposition does not yield a satisfactory \ROM in the sense that one wants to reduce the dimension further, one can resort to more general \MOR techniques. The goal is to find a \ROM that approximates the input-output behavior of the \FOM~\eqref{eq:pH:sys} as well as the Hamiltonian dynamic~\eqref{eq:HamiltonianDynamic}.
Conceptually, given some \FOM \eqref{eq:ph:ext} and a reduced dimension $\stateDimRed\ll\stateDim$, we want to find an extended \pH system 
\begin{equation}
	\label{eq:pHred:sys}
	\tilde \Sigma_{\mathsf{epH}}\quad \left\{\quad\begin{aligned}
		\dot{\stateRed}(t) &= (\reduce{J}-\reduce{R})\hamiltonianHessRed\stateRed(t) + (\reduce{G}+\reduce{P})\inpVar(t),\\
		\outVarRed(t) &= (\reduce{G}-\reduce{P})^\T \hamiltonianHessRed \stateRed(t) + (\reduce{S}-\reduce{N})\inpVar(t), \\
		\outVarHamRed(t) &= \tfrac{1}{2} \stateRed(t)^\T \hamiltonianHessRed \stateRed(t)
	\end{aligned}\right.
\end{equation}
of order $r$ that renders the approximation error $\|\Sigma_{\mathsf{epH}} - \tilde \Sigma_{\mathsf{epH}}\|_{\Htwo}$ as small as possible. 
Note that finding such a system is at least as challenging as attempting to solve the structure-preserving \MOR optimization problem
\begin{equation*}
	\min_{\Sigma} \|\Sigma_{\mathsf{pH}} - \systemPHRed\|_{\Htwo}^2 \quad\text{subject to}\quad \systemPHRed \text{ is \pH \LTI and of order $r$}.
\end{equation*}
Here, $\Sigma_{\mathsf{pH}}$ again represents the system corresponding to the linear output of~\eqref{eq:ph:ext}. Unfortunately, this problem is highly nonlinear and non-convex, so systematically finding globally optimal solutions is not possible in general.
Nevertheless, there are numerous techniques to find good (not necessarily optimal) reduced-order approximations of $\Sigma_{\mathsf{pH}}$. We refer to \Cref{subsec:literature} for an overview. These techniques are approved in practice; see~\cite[Sec.~8]{MehU22} and the references therein.
Building upon these techniques, we assume that a \ROM~\eqref{eq:pHred:sys} is available whose standard \LTI part $\tilde \Sigma_{\mathsf{pH}}$ constitutes a good approximation of $\Sigma_{\mathsf{pH}}$ in the usual sense.
It is crucial to note that we can replace the Hessian of the Hamiltonian~$\hamiltonianHessRed$ with any other positive definite solution of the \KYP inequality~\eqref{eq:KYP} without changing the input-output behavior of $\tilde \Sigma_{\mathsf{pH}}$; see \eqref{eqs:pH:decomposition} and the discussion thereafter. This degree of freedom motivates us to interpret the extended \ROM~$\tilde \Sigma_{\mathsf{epH}}$ as a function of $\hamiltonianHessRed$ and consider the minimization problem
\begin{equation}
	\label{eq:opt_problem_with_extended_norm}
		\min_{\hamiltonianHessRed\in\Spd{\stateDimRed}} \|\Sigma_{\mathsf{epH}} - \tilde \Sigma_{\mathsf{epH}}(\hamiltonianHessRed)\|_{\Htwo}^2 \quad\text{subject to}\quad \calW_{\systemRed}(\hamiltonianHessRed) \in \Spsd{\stateDimRed+\inpVarDim},
\end{equation}
which we call the \emph{energy matching optimization problem}.
In the following, we show that the latter optimization problem is convex and provide algorithmic means to solve it.

\subsection{Analysis of the energy matching optimization problem}
Since $\|\systemPH - \systemPHRed(\hamiltonianHessRed)\|_{\Htwo}^2$ is independent of $\hamiltonianHessRed\in\KYPset_{\systemRed}$, any minimizer of \eqref{eq:opt_problem_with_extended_norm} is a minimizer of
\begin{equation}
	\label{eq:energyErrMin}
	\min_{\hamiltonianHessRed\in\Spd{\stateDimRed}} \|\systemHam - \systemHamRed(\hamiltonianHessRed)\|_{\Htwo}^2 \quad\text{subject to} \quad \calW_{\reduce{\Sigma}}(\hamiltonianHessRed) \in \Spsd{\stateDimRed+\inpVarDim}
\end{equation}
and vice versa.
Using the discussion in \Cref{sec:eph}, we introduce the cost functional
\begin{equation}
    \label{eq:costFunctional}
    \objectiveFunc(\hamiltonianHessRed)  \vcentcolon= \|\systemHam - \systemHamRed(\hamiltonianHessRed)\|_{\Htwo}^2 = \tfrac{1}{4}\tr(\ctrlG \hamiltonianHess \ctrlG \hamiltonianHess)
    + \tfrac{1}{4}\tr(\reduce{\ctrlG} \reduce{\hamiltonianHess} \reduce{\ctrlG} \reduce{\hamiltonianHess}) - \tfrac{1}{2} \tr(Y^\T \hamiltonianHess Y \reduce{\hamiltonianHess})
\end{equation}
where $\ctrlG, \reduce{\ctrlG}$, and $\sylY$ are the unique solution of the linear matrix equations
\begin{align}
    \label{eq:costFunctional:matrixEquations}
    A\ctrlG + \ctrlG A^\T + BB^\T &= 0, &
    \reduce{A}\reduce{\ctrlG} + \reduce{\ctrlG} \reduce{A}^\T + \reduce{B}\reduce{B}^\T &= 0, &
  	A \sylY + \sylY \reduce{A}^\T + B \reduce{B}^\T &= 0.
\end{align}
\begin{example}
    \label{ex:energy-matching}
	To illustrate the optimization problem, we discuss~\eqref{eq:energyErrMin} with a concrete academic toy example. Suppose the \FOM~\eqref{eq:pH} is given by the matrices
	\begin{align*}
		J &= \begin{bmatrix}
            0 & 1\\
            -1 & 0
        \end{bmatrix}, &
		R &= \begin{bmatrix}
			2 & 0\\
			0 & 1
		\end{bmatrix}, &
		Q &= \begin{bmatrix}
			1 & 0\\
			0 & 1
		\end{bmatrix}, &
		A &= (J-R)Q = \begin{bmatrix}
			-2 & 1\\
			-1 & -1
		\end{bmatrix}, & 
		G &= \begin{bmatrix}
			6\\0
		\end{bmatrix}, &
		D &= 1.
	\end{align*}
	Accordingly, the controllability Gramian is $\ctrlG = \begin{smallbmatrix}
			8 & -2\\
			-2 & 2
	\end{smallbmatrix}$,
	and hence $\|\systemHam\|_{\calH_2}^2 =  \tfrac{1}{4}\tr(\ctrlG Q \ctrlG Q) = 19$. For the reduced model, we make the choice
	\begin{align}
		\label{eq:exEnergyMinimization:ROM}
		\reduce{A} &= -2, &
		\reduce{B} &= 6, &
		\reduce{C} &= 6, &
		\reduce{D} &= 1,
	\end{align}
	and we immediately see that the \KYP inequality
	\begin{displaymath}
    \calW_\system(\hamiltonianHessRed) =
		\begin{bmatrix}
			4\hamiltonianHessRed & 6-6\hamiltonianHessRed\\
			6-6\hamiltonianHessRed & 2
		\end{bmatrix}\succcurlyeq 0
    \end{displaymath}
    is satisfied for any $\hamiltonianHessRed \in [\tfrac{10}{9} - \tfrac{\sqrt{76}}{18},\tfrac{10}{9} + \tfrac{\sqrt{76}}{18}] = \KYPset_{\system}$. 
    In particular, the \ROM is passive and the optimization problem~\eqref{eq:energyErrMin} is feasible. The Gramian for the \ROM and the solutions of the Sylvester equation~\eqref{eq:costFunctional:matrixEquations} are
	\begin{equation*}
		\reduce{\ctrlG} = 9 \quad\text{and}\quad
		\sylY = \tfrac{1}{13}\begin{bmatrix}
			108\\-36
		\end{bmatrix}.
	\end{equation*}
	We thus obtain $\objectiveFunc(\hamiltonianHessRed) = 19 + \tfrac{81}{4}\hamiltonianHessRed^2 - 2\cdot\tfrac{3240}{169}\hamiltonianHessRed$.
	The first-order necessary condition implies $\hamiltonianHessRed^\star = \tfrac{160}{169} \approx 0.95$, which is an element of the feasible set and thus the optimal point. 	
\end{example}

\begin{remark}	
	We emphasize that the \ROM~\eqref{eq:exEnergyMinimization:ROM} in \Cref{ex:energy-matching} is obtained via Galerkin projection onto the space spanned by the matrix $V = [1,0]^\T$, which in this particular setting preserves the \pH-structure. Nevertheless, we have $\hamiltonianHessRed^\star \neq V^\T \hamiltonianHess V = 1$, i.e., a standard projection framework does not automatically yield the best Hamiltonian in the \ROM. Moreover, the optimal Hamiltonian is not an element of the solution set of the \ARE~\eqref{eq:passivity-riccati} for the \ROM, which is $\{\tfrac{10}{9} - \tfrac{\sqrt{76}}{18},\tfrac{10}{9} + \tfrac{\sqrt{76}}{18}\}$.
\end{remark}
We make the following observations.
\begin{lemma}\label{lem:gradient}
    Assume that $\Sigma$ is asymptotically stable and $\reduce{\Sigma}$ is minimal and asymptotically stable. Then $\objectiveFunc\colon\Spd{\stateDimRed}\to\R$ as defined in~\eqref{eq:costFunctional} is twice differentiable, strictly convex, and 
    \begin{equation}
        \label{eq:gradient}
        \nabla_{\hamiltonianHessRed} \objectiveFunc(\hamiltonianHessRed) = \tfrac{1}{2} \left(\reduce{\ctrlG} \hamiltonianHessRed \reduce{\ctrlG} - \sylY^\T \hamiltonianHess \sylY\right).
    \end{equation}
\end{lemma}

\begin{proof}
    We first derive the first order derivative of $\objectiveFunc$ with respect to $\hamiltonianHessRed$.
    Since $X$ is symmetric, we obtain
    \begin{equation*}
        \frac{\partial}{\partial X} \tr(AXAX) = 2AXA^\T \quad \mathrm{and} \quad \frac{\partial}{\partial X} \tr(AX) = A.
    \end{equation*}
    With this, we immediately obtain
    \begin{equation*}
        \frac{\partial}{\partial \hamiltonianHessRed} \tfrac{1}{4}\tr(\ctrlG \hamiltonianHess \ctrlG \hamiltonianHess) + \tfrac{1}{4}\tr(\reduce{\ctrlG} \reduce{\hamiltonianHess} \reduce{\ctrlG} \reduce{\hamiltonianHess}) - \tfrac{1}{2}\tr(Y^\T \hamiltonianHess Y \reduce{\hamiltonianHess}) = \tfrac{1}{2}\left(\ctrlG \hamiltonianHessRed \ctrlG - Y^\T \hamiltonianHess Y \right).
    \end{equation*}
    Furthermore, making use of the vectorization of the first derivative, we obtain the second derivative as 
    \begin{equation*}
        \begin{aligned}
            \frac{\partial}{\partial \vectorize(\hamiltonianHessRed)} \vectorize\left(\tfrac{1}{2}\reduce{\ctrlG} \hamiltonianHessRed \reduce{\ctrlG} - \tfrac{1}{2}\sylY^\T \hamiltonianHess \sylY\right) &= \frac{\partial}{\partial \vectorize(\hamiltonianHessRed)}
            \tfrac{1}{2} \left(\reduce{\ctrlG} \otimes \reduce{\ctrlG}\right) \vectorize(\hamiltonianHessRed) 
            - \tfrac{1}{2}\left(\sylY \otimes \sylY\right) \vectorize(\hamiltonianHess)\\
            &= \tfrac{1}{2} \left(\reduce{\ctrlG} \otimes \reduce{\ctrlG}\right),
        \end{aligned}
    \end{equation*}
    which is strictly positive whenever  $\reduce{\Sigma}$ is minimal and asymptotically stable. Hence, $\objectiveFunc$ is strictly convex.
\end{proof}

\begin{theorem}\label{thm:energyOptimizationUniqueSolution}
    In addition to the assumptions of \Cref{lem:gradient} suppose that $\systemRed$ is passive. Then the optimization problem~\eqref{eq:energyErrMin} is solvable and has a unique solution.
\end{theorem}

\begin{proof}
    Since $\reduce{\Sigma}$ is minimal and passive, \Cref{thm:KYP}\,\ref{thm:KYP:posDef} implies the existence of $\smash{\hamiltonianHessRed\in\KYPset_{\systemRed}}$. Moreover,~$\objectiveFunc$ is bounded from below. Let $(\hamiltonianHessRed_k)_{k\in\N}$ denote a sequence in $\smash{\KYPset_{\systemRed}}$ such that 
    \begin{displaymath}
        \lim_{k\to\infty} \objectiveFunc(\hamiltonianHessRed_k) = \inf_{X\in\KYPset_{\reduce{\Sigma}}} \objectiveFunc(X).
    \end{displaymath} 
    Since $\smash{\KYPset_{\systemRed}}$ is bounded, cf.~\Cref{thm:KYP}\,\ref{thm:KYP:bounded}, we can choose a convergent subsequence $(\hamiltonianHessRed_{k_j})_{j\in\N}$ with limit $\hamiltonianHessRed^\star \vcentcolon= \lim_{j\to\infty} \hamiltonianHessRed_{k_j}$. By construction, we obtain $\smash{\calW_{\systemRed}}(\hamiltonianHessRed^\star)\in\Spsd{\stateDimRed+\inpVarDim}$ such that \Cref{thm:KYP}\,\ref{thm:KYP:posDef} implies $\hamiltonianHessRed^\star\in\KYPset_{\systemRed}$. The continuity of $\objectiveFunc$ (cf.~\Cref{lem:gradient}) thus implies that~$\hamiltonianHessRed^\star$ is a minimizer of~\eqref{eq:energyErrMin}. The uniqueness follows from the strict convexity of $\objectiveFunc$; cf.~\Cref{lem:gradient}.
\end{proof}

\subsection{A special case: Positive-real balanced truncation}
\label{sec:specialcase-prbt}
To obtain further insights into the optimization problem~\eqref{eq:energyErrMin}, we consider the particular case that the \ROM is obtained via \PRBT and the Hessian of the Hamiltonian of the \FOM is given by the minimal solution of the \KYP inequality. In this case, the minimal solution of the \KYP inequality of the \FOM is given via projection of the minimal solution of the \FOM \KYP, and hence, one might get the idea that in this specific scenario, \PRBT is optimal with respect to~\eqref{eq:energyErrMin}. We refer to the forthcoming numerical \Cref{subsec:MSDxMINHam} for a corresponding numerical example. The following two toy examples, generated via the balanced parametrization for positive-real systems from~\cite{Obe91}, demonstrate that \PRBT can be optimal in the setting described above (cf.~\Cref{ex:prbt:ex}), but in general, there is no guarantee; see \Cref{ex:prbt:counter-ex}.

\begin{example}
	\label{ex:prbt:ex}
    Consider the system described by the matrices
    \begin{equation*}
        \begin{aligned}
            A &= \begin{bmatrix}
                -2 & -4
                \\
                -4 & -9
            \end{bmatrix}, &
            B &= \begin{bmatrix}
                4 \\
                4
            \end{bmatrix}, &
            C &= \begin{bmatrix}
                4 & 4
            \end{bmatrix}, &
            D &= 1, &
            \hamiltonianHess_{\min} &= \begin{bmatrix}
            	\tfrac{1}{2} & 0\\
            	0 & \tfrac{1}{4}
            \end{bmatrix},
        \end{aligned}
    \end{equation*}
    which is already in positive-real balanced form~\cite{Obe91} with Gramians $\hamiltonianHess_{\min}$. Then the \ROM obtained by \PRBT of order $\stateDimRed=1$ is given by the upper left entry, i.e.,
    \begin{equation}
    	\label{eq:prbtToy:optimalROM}
        \begin{aligned}
            \reduce{A} &= -2, & \reduce{B} &= 4, & \reduce{C} &= 4, & \reduce{D} &= 1, & \hamiltonianHessRed_{\min} &= \tfrac{1}{2}.
        \end{aligned}
    \end{equation}
    We obtain $\reduce{\ctrlG} = 4$, $\sylY = \begin{smallbmatrix}  4\\0 \end{smallbmatrix}$, such that for any $\hamiltonianHessRed\in\KYPset_{\systemRed} = [\tfrac{1}{2},2]$ the cost functional~\eqref{eq:costFunctional} is given by
    \begin{align*}
        \objectiveFunc(\hamiltonianHessRed) &= \|\systemQO\|_{\Htwo} + 4 \hamiltonianHessRed^2 - 4 \reduce{Q}, &
        \nabla\objectiveFunc(\hamiltonianHessRed) &= 8 \hamiltonianHessRed - 4,
    \end{align*}
    which is minimized for $\hamiltonianHessRed^\star = \frac{1}{2} = \hamiltonianHessRed_{\min}$, i.e., the \PRBT \ROM~\eqref{eq:prbtToy:optimalROM} is optimal.
\end{example}

\begin{example}
	\label{ex:prbt:counter-ex}
	Consider the positive-real balanced system
    \begin{equation*}
        \begin{aligned}
            A &= \begin{bmatrix}
                -1 & -\frac{9}{2}
                \\
                -\frac{9}{2} & -27
            \end{bmatrix}, &
            B &= \begin{bmatrix}
                4 \\
                4
            \end{bmatrix}, &
            C &= \begin{bmatrix}
                4 & 4
            \end{bmatrix}, &
            D &= \frac{1}{3}, &
            \hamiltonianHess_{\min} &= \begin{bmatrix}
            	\tfrac{3}{4} & 0\\
            	0 & \tfrac{1}{4}
            \end{bmatrix},
        \end{aligned}
    \end{equation*}
   	with diagonal Gramian $\hamiltonianHess_{\min}$. The one-dimensional \PRBT \ROM is given by 
    \begin{equation*}
        \begin{aligned}
            \reduce{A} &= -1, & 
            \reduce{B} &= 4, & 
            \reduce{C} &= 4, & 
            \reduce{D} &= \tfrac{1}{3}, &
            \hamiltonianHessRed_{\min} &= \tfrac{3}{4}.
        \end{aligned}
    \end{equation*}
	We obtain $\reduce{\ctrlG} = 8$, $\sylY = -\tfrac{16}{143} \begin{smallbmatrix}  -94\\\phantom{-}10 \end{smallbmatrix}$, such that for any given $\hamiltonianHessRed\in\KYPset_{\systemRed} = [\tfrac{3}{4},\tfrac{4}{3}]$ the cost functional~\eqref{eq:costFunctional} is given by
    \begin{align*}
        \objectiveFunc(\hamiltonianHessRed) &= \|\systemQO\|_{\Htwo} + 16 \hamiltonianHessRed^2 - \tfrac{851456}{143^2}\reduce{Q}, &
        \nabla\objectiveFunc(\hamiltonianHessRed) &= 32 \hamiltonianHessRed - \tfrac{851456}{143^2}.
    \end{align*}
    We deduce $\hamiltonianHessRed_{\min}<\hamiltonianHessRed^\star = \tfrac{26608}{143^2} \in \KYPset_{\reduce{\system}}$ and thus conclude that the reduced Hamiltonian $\hamiltonianHessRed_{\min}$  obtained via \PRBT is not optimal.
\end{example}

\subsection{Numerical approach}
The optimization problem~\eqref{eq:energyErrMin} is a standard \emph{semi-definite program} (\SDP), more precisely, a \emph{linear matrix inequality} (\LMI) problem. 
For this class of problems, efficient solvers exist, such as Mosek~\cite{mosek} or SeDuMi~\cite{Stu99}. So the first and most straightforward strategy is to apply this type of solver to the optimization problem~\eqref{eq:energyErrMin}.
However, it is important to consider the limitation of these solvers: They can face challenges with scalability or numerical stability, especially when dealing with ill-conditioned problems, as discussed in the forthcoming~\Cref{sec:num-exp:poro}. Therefore, we propose a second strategy, another classical method for solving \LMI problems, namely the barrier (or path following) method~\cite{BoyV04}.

Instead of solving the constraint optimization problem~\eqref{eq:energyErrMin} directly, the idea of the barrier method is to solve a sequence of unconstrained optimization problems, where the semi-definite constraint of~\eqref{eq:energyErrMin} is realized by a barrier function.
The barrier function assures that the iterates stay in the feasible set and is defined as
\begin{equation}
    \label{eq:barrierFunction}
    \psi\colon\R^{\stateDimRed\times\stateDimRed}\to \overline{\R},\quad \hamiltonianHessRed \mapsto \begin{cases}
    -\ln \det \left(\calW_{\reduce{\system}_{\mathrm{pH}}}(\hamiltonianHessRed)\right), & \text{if } \calW_{\reduce{\system}_{\mathrm{pH}}}(\hamiltonianHessRed)\in\Spd{\stateDimRed+\inpVarDim},\\
    \infty, & \text{otherwise}.
    \end{cases}
\end{equation}
Now, for $\alpha>0$ the parametrized objective function is given by $\objectiveFunc_{\alpha,\psi}(\hamiltonianHessRed) \vcentcolon= \objectiveFunc(\hamiltonianHessRed) + \alpha\psi(\hamiltonianHessRed)$ with the corresponding (unconstrained) optimization problem
\begin{equation}
    \label{eq:energyErrMinWithBarrier}
    \min_{\hamiltonianHessRed\in\Spsd{\stateDimRed}} \objectiveFunc_{\alpha,\psi}(\hamiltonianHessRed).
\end{equation}
Note that the barrier function~\eqref{eq:barrierFunction} requires~\eqref{eq:pH:rom} to be strictly passive. If this is not the case, then a perturbation of the feedthrough term is required (see the forthcoming \Cref{sec:num-exp}).

\begin{proposition}
	\label{prop:gradientBarrierFunction}
	Assume that the \ROM~\eqref{eq:pHred:sys} is passive and let $X\in\KYPset_{\systemRed}$ with $\det(\calW_{\systemPHRed}(X))>0$. Then the gradient of the barrier function is given by   
    \begin{align*}
        \nabla_X \ln\left(\det\left(\calW_{\systemPHRed}(X)\right)\right) =
        \begin{bmatrix}
            -\reduce{A} & -\reduce{B}
        \end{bmatrix}
        \inv{\left(\calW_{\systemPHRed}(X)\right)}
        \begin{bmatrix}
            I \\
            0
        \end{bmatrix}
        + 
        \begin{bmatrix}
            I & 0
        \end{bmatrix}
        \inv{\left(\calW_{\systemPHRed}(X)\right)}
        \begin{bmatrix}
            -\reduce{A}^\T \\
            -\reduce{B}^\T
        \end{bmatrix}.
    \end{align*}
\end{proposition}

\begin{proof}
    The proof is given in~\Cref{app:proof-prop:gradientBarrierFunction}.
\end{proof}
Now, for a decreasing sequence of~$\alpha_k$, we minimize~\eqref{eq:energyErrMinWithBarrier} with a gradient-based optimization method such as a quasi-Newton method. 
Since the surrogate model is assumed to be stable, \Cref{thm:KYP} implies that any solution of the \KYP inequality is positive definite. Hence, the barrier function automatically ensures that $\hamiltonianHessRed$ is symmetric positive definite whenever $\objectiveFunc_{\alpha,\psi}(\hamiltonianHessRed)$ is finite.

In our numerical implementation, we reduce the degrees of freedom by explicitly forcing $\hamiltonianHessRed$ to be symmetric via the \emph{half-vectorization} operator $\vech\colon\R^{\stateDimRed\times\stateDimRed}\to\R^{\stateDimRed(\stateDimRed+1)/2}$; see for instance \cite{MagN19}. In this way, we can represent a $\stateDimRed\times\stateDimRed$ symmetric matrix as a vector of length $\stateDimRed(\stateDimRed+1)/2$ and vice versa. The relation to the standard vectorization is given by
\begin{equation}
    \label{eq:duplicationMatrix}
    \vectorize(A) = \mathcal{D}_r \vech(A),
\end{equation}
with the duplication matrix $\mathcal{D}_r\in\R^ {\stateDimRed^2\times\stateDimRed(\stateDimRed+1)/2}$ \cite{MagN19}. We make the following technical observation to compute the gradient of the barrier method with respect to the half-vectorized reduced order Hessian of the Hamiltonian.

\begin{lemma}
    \label{lem:vech-chain-rule}
    Consider $f\colon\R^{\stateDimRed\times\stateDimRed}\to\R$. Then
    \begin{equation*}
        \frac{\partial}{\partial x} f(\vech^{-1}(x)) = \mathcal{D}_r^\T \frac{\partial}{\partial \vech^{-1}(x)} f(\vech^{-1}(x))
    \end{equation*}
    for any $x\in\R^{\stateDimRed(\stateDimRed+1)/2}$
\end{lemma}
\begin{proof}
    Relation~\eqref{eq:duplicationMatrix} yields $\frac{\partial}{\partial x} \vectorize(\vech^{-1}(x)) = \frac{\partial}{\partial x} \mathcal{D}_r \vech(\vech^{-1}(x)) = \mathcal{D}_r$.
    Applying the chain rule yields the desired result.
\end{proof}

Let $\reduce{q} \vcentcolon= \vech(\hamiltonianHessRed)$. Then, \Cref{lem:vech-chain-rule} implies
\begin{equation*}
    \nabla_{\reduce{q}} \objectiveFunc_{\alpha,\psi}(\vech^{-1}(\reduce{q})) = 
    \mathcal{D}_r^\T \vectorize(\nabla_{\vech^{-1}(\reduce{q})} \objectiveFunc_{\alpha,\psi}(\vech^{-1}(\reduce{q}))).
\end{equation*}
Hence, we can compute the gradient using \Cref{lem:gradient} and \Cref{prop:gradientBarrierFunction}.
The resulting algorithm is described in~\Cref{alg:enermatch}. 

\begin{algorithm}
	\caption{Energy matching barrier method}
	\label{alg:enermatch}
	\KwIn{\FOM~\eqref{eq:pH}, passive \ROM~\eqref{eq:LTI:rom}, initial Hamiltonian $\hamiltonianHessRed_0$ with $\det(\calW_{\systemPHRed}(\hamiltonianHessRed_0))>0$}
	\KwOut{Approximate minimizer $\hamiltonianHessRed_{\mathsf{opt}}\in\Spsd{\stateDimRed}$ of~\eqref{eq:energyErrMinWithBarrier}}
    Set $\reduce{q}_0 \vcentcolon= \vech(\hamiltonianHessRed_0)$.\\
	\For{$\alpha \in \{10^{-3}, 10^{-4}, \dots, 10^{-15}\}$}{
    Set $f(\reduce{q}) \vcentcolon= \objectiveFunc_{\alpha, \psi}(\vech^{-1}(\reduce{q}))$.\\
    Solve $\reduce{q}_{\mathsf{opt}} \vcentcolon= \argmin_{\reduce{q}} f(\reduce{q})$ with initial value $\reduce{q}_{0}$.\\
    Set $\reduce{q}_{0} \vcentcolon= \reduce{q}_{\mathsf{opt}}$.\\
}
\Return{$\hamiltonianHessRed_\mathsf{opt} \vcentcolon= \vech^{-1}(\reduce{q}_{\mathsf{opt}})$}
\end{algorithm}

\subsection{Model reduction workflow}
\label{sec:workflow}

Building on the concepts discussed in~\Cref{sec:minimality} and~\Cref{sec:energy-matching}, we propose the following workflow for constructing \ROMs that effectively capture both the input-output dynamic and the Hamiltonian dynamic of the \FOM:
\begin{enumerate}[label=\textbf{Step \arabic*},leftmargin=*,align=left]
\item \label{wf:minreal} Compute a minimal realization of the \FOM using the structure-preserving Kalman-like decomposition (see \Cref{thm:KalmanFormPHQO} und \Cref{cor:minreal}).
\item \label{wf:iorom} Construct a \ROM that approximates the input-output dynamic using any passivity-preserving \MOR method.
\item \label{wf:hamrom} Perform energy matching (\Cref{alg:enermatch}) as a post-processing step to determine the optimal $\reduce{Q}\in\KYPset_{\systemRed}$ with respect to the Hamiltonian dynamic of the minimal realization.
\end{enumerate}

\section{Numerical experiments}\label{sec:num-exp}
In the following sections, we illustrate the effectiveness of the previously described workflow on three well-established \pH benchmark systems: a RCL ladder network, a mass-spring-damper model and a poroelasticity model. 
These systems often serve as benchmark examples in research articles on (\pH) \MOR.
The mass-spring-damper model and the poroelasticity model are accessible through the \pH benchmark systems collection\footnote{\url{https://github.com/Algopaul/PortHamiltonianBenchmarkSystems.jl}}.
Details on the RCL ladder network can be found in~\cite{PolS10}.
We utilize the default parametrization for each system.
The first example highlights the effectiveness of the Kalman-like decomposition applied to the RCL ladder network, utilizing only~\ref{wf:minreal} from the workflow.
In contrast, the full workflow is applied in the subsequent examples, with \ROMs constructed in~\ref{wf:iorom} using \PRBT and \pHIRKA.

Regarding the implementation details of the methods, the following remarks are in order:
\begin{inparaenum}[i)]
    \item As in~\cite{BreU22}, we use the minimal solution of the \KYP inequality~\eqref{eq:KYP} as the Hamiltonian to obtain a \pH representation for the \ROMs from \PRBT.
    \item To make the computation of the extremal solutions of the \KYP inequality~\eqref{eq:KYP}, i.e., the stabilizing and anti-stabilizing solution of the \ARE~\eqref{eq:passivity-riccati} feasible, we add an artificial feedthrough term $D= \num{1e-6} I_{\inpVarDim}$ to the benchmark systems.
    \item For solving linear matrix equations as well as \AREs we use the \texttt{MatrixEquations.jl}\footnote{\url{https://github.com/andreasvarga/MatrixEquations.jl}} package.
    \item The computation of the $\Htwo$-norm for standard \LTI systems is done via the julia package \texttt{ControlSystems.jl}~\cite{CarFFHT22}.
    \item For the implementation details for \pHIRKA and \PRBT, we refer to~\cite{BreU22}.
    \item We formulate the \SDP problem within the JuMP framework (\texttt{JuMP.jl})~\cite{LubDGHBV23}, which supports interfaces to many open-source and commercial solvers.
    \item For the minimization of~\eqref{eq:energyErrMinWithBarrier} we use the BFGS implementation from \texttt{Optim.jl}~\cite{MogR18}.
    \item To initialize \Cref{alg:enermatch}, we pick $\hamiltonianHessRed_0$ as the optimal solution of the optimization problem~\eqref{eq:energyErrMin}, where we replace the feasible set (solutions of the \KYP inequality) with the solutions of the \ARE~\eqref{eq:passivity-riccati}. Note that the resulting \KYP matrix $\calW_{\systemPHRed}(\hamiltonianHessRed_0)$ is rank deficient by construction and hence perturbed to render it positive definite.
\end{inparaenum}

\vspace{0.2cm}
\noindent\fbox{%
    \parbox{0.98\textwidth}{%
        The code and data used to generate the subsequent results are accessible via
		\begin{center}
			\href{http://doi.org/10.5281/zenodo.8335231}{doi:10.5281/zenodo.8335231}
		\end{center}
		under MIT Common License.
    }%
}
\vspace{0.2cm}

\subsection{RCL ladder network}

In this example, we analyze a single-input single-output RCL ladder network from~\cite{PolS10} with $\stateDim = 5000$. Using our structure-preserving Kalman-like decomposition, we obtain a (numerically) minimal realization of order $\stateDimRed = 55$, achieving a relative $\mathcal{H}_2$-error on the order of $\num{1e-8}$ for both the input-output and Hamiltonian dynamic. This result highlights the effectiveness of the Kalman-like decomposition for this example. 

\subsection{Mass-spring-damper system}

Our second example considers a \pH mass-spring-damper system with $\stateDim=100$ degrees of freedom and the input/output dimension $\inpVarDim=2$. The system was introduced in~\cite{GugPBV12} and is described in detail in the \pH benchmark systems collection. 
A minimal realization of order $\stateDim=77$ is obtained, with a negligible $\Htwo$-error for both the input-output dynamic and the Hamiltonian dynamic.
The $\Htwo$-error for the input-output dynamic and the Hamiltonian dynamic over the reduced orders $\stateDimRed=2,4,\ldots,20$ is then shown for the respective methods in \Cref{fig:msd-h2}.
Comparing the $\Htwo$-errors of the input-output dynamic, it can be observed that \pHIRKA leads to an approximation error that is, in general, a few orders of magnitude worse compared to \PRBT (as already observed in \cite{SchV20, BreU22}). However, for the Hamiltonian dynamic, it is the other way around; \pHIRKA yields significantly better approximations than \PRBT, which motivates performing our energy matching algorithm (\ref{wf:hamrom}) to improve the approximation of the Hamiltonian dynamic of the \ROMs obtained by \PRBT.
Using either \Cref{alg:enermatch} or the solution of the \SDP solver (which gives approximately the same result in this example), we can significantly improve the error of the Hamiltonian dynamic of \PRBT (see \EMPRBT in \Cref{fig:msd-h2}). 
After the optimization, the $\Htwo$-error of the Hamiltonian dynamic is comparable to, and for some reduced orders even better than, that achieved with \pHIRKA, while maintaining the approximation quality of the input-output dynamic.
\PRBT, in combination with the energy matching algorithm, yields \ROMs that achieve the initial goal.
In \Cref{fig:msd-h2}, we also plot the $\mathcal{H}_2$-errors of \ROMs obtained by setting the Hessian of the Hamiltonian to the optimal rank-minimizing solution of the \KYP inequality, i.e., the solutions of the \ARE~\eqref{eq:passivity-riccati}.
These \ROMs are denoted with $\PRBT(X^\star)$ and only slightly improve the Hamiltonian dynamic $\Htwo$-error of \PRBT, which shows that this choice is not sufficient and further optimization is required.

\begin{figure}[htpb]
    \centering
    \ref{leg:msd-h2}
    \resizebox{0.99\textwidth}{!}{
    \begin{tabular}{cc}

\begin{tikzpicture}[/tikz/background rectangle/.style={fill={rgb,1:red,1.0;green,1.0;blue,1.0}, fill opacity={1.0}, draw opacity={1.0}}, show background rectangle]
\begin{axis}[point meta max={nan}, point meta min={nan}, legend cell align={left}, legend columns={4}, title={}, title style={at={{(0.5,1)}}, anchor={south}, font={{\fontsize{14 pt}{18.2 pt}\selectfont}}, color={rgb,1:red,0.0;green,0.0;blue,0.0}, draw opacity={1.0}, rotate={0.0}, align={center}}, legend style={color={rgb,1:red,0.0;green,0.0;blue,0.0}, draw opacity={1.0}, line width={1}, solid, fill={rgb,1:red,1.0;green,1.0;blue,1.0}, fill opacity={1.0}, text opacity={1.0}, font={{\fontsize{8 pt}{10.4 pt}\selectfont}}, text={rgb,1:red,0.0;green,0.0;blue,0.0}, cells={anchor={center}}, at={(0.02, 0.02)}, anchor={south west}}, axis background/.style={fill={rgb,1:red,1.0;green,1.0;blue,1.0}, opacity={1.0}}, anchor={north west}, xshift={1.0mm}, yshift={-1.0mm}, width={0.43\textwidth}, height={0.2866666666666667\textwidth}, scaled x ticks={false}, xlabel={Reduced order $r$}, x tick style={color={rgb,1:red,0.0;green,0.0;blue,0.0}, opacity={1.0}}, x tick label style={color={rgb,1:red,0.0;green,0.0;blue,0.0}, opacity={1.0}, rotate={0}}, xlabel style={at={(ticklabel cs:0.5)}, anchor=near ticklabel, at={{(ticklabel cs:0.5)}}, anchor={near ticklabel}, font={{\fontsize{11 pt}{14.3 pt}\selectfont}}, color={rgb,1:red,0.0;green,0.0;blue,0.0}, draw opacity={1.0}, rotate={0.0}}, xmajorgrids={true}, xmin={2}, xmax={20}, xticklabels={{$2$,$4$,$6$,$8$,$10$,$12$,$14$,$16$,$18$,$20$}}, xtick={{2.0,4.0,6.0,8.0,10.0,12.0,14.0,16.0,18.0,20.0}}, xtick align={inside}, xticklabel style={font={{\fontsize{8 pt}{10.4 pt}\selectfont}}, color={rgb,1:red,0.0;green,0.0;blue,0.0}, draw opacity={1.0}, rotate={0.0}}, x grid style={color={rgb,1:red,0.0;green,0.0;blue,0.0}, draw opacity={0.1}, line width={0.5}, solid}, axis x line*={left}, x axis line style={color={rgb,1:red,0.0;green,0.0;blue,0.0}, draw opacity={1.0}, line width={1}, solid}, scaled y ticks={false}, ylabel={Error $\|\Sigma_{\mathrm{pH}} - \widetilde{\Sigma}_{\mathrm{pH}}\|_{\mathcal{H}_2}$}, y tick style={color={rgb,1:red,0.0;green,0.0;blue,0.0}, opacity={1.0}}, y tick label style={color={rgb,1:red,0.0;green,0.0;blue,0.0}, opacity={1.0}, rotate={0}}, ylabel style={at={(ticklabel cs:0.5)}, anchor=near ticklabel, at={{(ticklabel cs:0.5)}}, anchor={near ticklabel}, font={{\fontsize{11 pt}{14.3 pt}\selectfont}}, color={rgb,1:red,0.0;green,0.0;blue,0.0}, draw opacity={1.0}, rotate={0.0}}, ymode={log}, log basis y={10}, ymajorgrids={true}, ymin={1.0e-6}, ymax={1.0}, yticklabels={{$10^{-6}$,$10^{-5}$,$10^{-4}$,$10^{-3}$,$10^{-2}$,$10^{-1}$,$10^{0}$}}, ytick={{1.0e-6,1.0e-5,0.0001,0.001,0.01,0.1,1.0}}, ytick align={inside}, yticklabel style={font={{\fontsize{8 pt}{10.4 pt}\selectfont}}, color={rgb,1:red,0.0;green,0.0;blue,0.0}, draw opacity={1.0}, rotate={0.0}}, y grid style={color={rgb,1:red,0.0;green,0.0;blue,0.0}, draw opacity={0.1}, line width={0.5}, solid}, axis y line*={left}, y axis line style={color={rgb,1:red,0.0;green,0.0;blue,0.0}, draw opacity={1.0}, line width={1}, solid}, colorbar={false}, legend to name={leg:msd-h2}]
    \addplot[color={rgb,1:red,0.0078;green,0.2431;blue,1.0}, name path={5}, draw opacity={1.0}, line width={1}, solid, mark={square*}, mark size={3.0 pt}, mark repeat={1}, mark options={color={rgb,1:red,0.0078;green,0.2431;blue,1.0}, draw opacity={1.0}, fill={rgb,1:red,0.0078;green,0.2431;blue,1.0}, fill opacity={1.0}, line width={0.75}, rotate={0}, solid}]
        table[row sep={\\}]
        {
            \\
            2.0  0.2901463818982096  \\
            4.0  0.19784072647452658  \\
            6.0  0.16849533683660167  \\
            8.0  0.1176829751408123  \\
            10.0  0.085824887592962  \\
            12.0  0.062338892034003425  \\
            14.0  0.04579846175059834  \\
            16.0  0.03569407427232376  \\
            18.0  0.024738048324711007  \\
            20.0  0.016795196663010877  \\
        }
        ;
    \addlegendentry {$\textsf{pHIRKA}$}
    \addplot[color={rgb,1:red,0.102;green,0.7882;blue,0.2196}, name path={6}, draw opacity={1.0}, line width={1}, solid, mark={*}, mark size={3.0 pt}, mark repeat={1}, mark options={color={rgb,1:red,0.102;green,0.7882;blue,0.2196}, draw opacity={1.0}, fill={rgb,1:red,0.102;green,0.7882;blue,0.2196}, fill opacity={1.0}, line width={0.75}, rotate={0}, solid}]
        table[row sep={\\}]
        {
            \\
            2.0  0.49915504584281944  \\
            4.0  0.35066488124606654  \\
            6.0  0.11721032880126778  \\
            8.0  0.011027957530240922  \\
            10.0  0.0031763365319785056  \\
            12.0  0.0004229462236560842  \\
            14.0  8.901536017041567e-5  \\
            16.0  2.7295813857960007e-5  \\
            18.0  8.13736789840395e-6  \\
            20.0  6.723446454983663e-6  \\
        }
        ;
    \addlegendentry {$\textsf{PRBT}$}
    \addplot[color={rgb,1:red,0.5451;green,0.1686;blue,0.8863}, name path={7}, draw opacity={1.0}, line width={1}, dashed, mark={star}, mark size={3.0 pt}, mark repeat={1}, mark options={color={rgb,1:red,0.5451;green,0.1686;blue,0.8863}, draw opacity={1.0}, fill={rgb,1:red,0.5451;green,0.1686;blue,0.8863}, fill opacity={1.0}, line width={0.75}, rotate={0}, solid}]
        table[row sep={\\}]
        {
            \\
            2.0  0.4991550458428195  \\
            4.0  0.35066488124606643  \\
            6.0  0.11721032880126765  \\
            8.0  0.011027957530241016  \\
            10.0  0.00317633653197854  \\
            12.0  0.000422946223656625  \\
            14.0  8.901536016997153e-5  \\
            16.0  2.7295813860793007e-5  \\
            18.0  8.137367898745826e-6  \\
            20.0  6.723446452907418e-6  \\
        }
        ;
    \addlegendentry {$\textsf{PRBT}(X^\star)$}
    \addplot[color={rgb,1:red,0.9098;green,0.0;blue,0.0431}, name path={8}, draw opacity={1.0}, line width={1}, dashed, mark={diamond*}, mark size={3.0 pt}, mark repeat={1}, mark options={color={rgb,1:red,0.9098;green,0.0;blue,0.0431}, draw opacity={1.0}, fill={rgb,1:red,0.9098;green,0.0;blue,0.0431}, fill opacity={1.0}, line width={0.75}, rotate={0}, solid}]
        table[row sep={\\}]
        {
            \\
            2.0  0.49915504584281944  \\
            4.0  0.3506648812460665  \\
            6.0  0.11721032880126797  \\
            8.0  0.011027957530241035  \\
            10.0  0.0031763365319785546  \\
            12.0  0.0004229462236560733  \\
            14.0  8.901536017040462e-5  \\
            16.0  2.7295813857953224e-5  \\
            18.0  8.137367898403557e-6  \\
            20.0  6.723446454989155e-6  \\
        }
        ;
    \addlegendentry {$\textsf{EM-PRBT}$}
\end{axis}
\end{tikzpicture}
        &

\begin{tikzpicture}[/tikz/background rectangle/.style={fill={rgb,1:red,1.0;green,1.0;blue,1.0}, fill opacity={1.0}, draw opacity={1.0}}, show background rectangle]
\begin{axis}[point meta max={nan}, point meta min={nan}, legend cell align={left}, legend columns={1}, title={}, title style={at={{(0.5,1)}}, anchor={south}, font={{\fontsize{14 pt}{18.2 pt}\selectfont}}, color={rgb,1:red,0.0;green,0.0;blue,0.0}, draw opacity={1.0}, rotate={0.0}, align={center}}, legend style={color={rgb,1:red,0.0;green,0.0;blue,0.0}, draw opacity={1.0}, line width={1}, solid, fill={rgb,1:red,1.0;green,1.0;blue,1.0}, fill opacity={1.0}, text opacity={1.0}, font={{\fontsize{8 pt}{10.4 pt}\selectfont}}, text={rgb,1:red,0.0;green,0.0;blue,0.0}, cells={anchor={center}}, at={(1.02, 1)}, anchor={north west}}, axis background/.style={fill={rgb,1:red,1.0;green,1.0;blue,1.0}, opacity={1.0}}, anchor={north west}, xshift={1.0mm}, yshift={-1.0mm}, width={0.43\textwidth}, height={0.2866666666666667\textwidth}, scaled x ticks={false}, xlabel={Reduced order $r$}, x tick style={color={rgb,1:red,0.0;green,0.0;blue,0.0}, opacity={1.0}}, x tick label style={color={rgb,1:red,0.0;green,0.0;blue,0.0}, opacity={1.0}, rotate={0}}, xlabel style={at={(ticklabel cs:0.5)}, anchor=near ticklabel, at={{(ticklabel cs:0.5)}}, anchor={near ticklabel}, font={{\fontsize{11 pt}{14.3 pt}\selectfont}}, color={rgb,1:red,0.0;green,0.0;blue,0.0}, draw opacity={1.0}, rotate={0.0}}, xmajorgrids={true}, xmin={2}, xmax={20}, xticklabels={{$2$,$4$,$6$,$8$,$10$,$12$,$14$,$16$,$18$,$20$}}, xtick={{2.0,4.0,6.0,8.0,10.0,12.0,14.0,16.0,18.0,20.0}}, xtick align={inside}, xticklabel style={font={{\fontsize{8 pt}{10.4 pt}\selectfont}}, color={rgb,1:red,0.0;green,0.0;blue,0.0}, draw opacity={1.0}, rotate={0.0}}, x grid style={color={rgb,1:red,0.0;green,0.0;blue,0.0}, draw opacity={0.1}, line width={0.5}, solid}, axis x line*={left}, x axis line style={color={rgb,1:red,0.0;green,0.0;blue,0.0}, draw opacity={1.0}, line width={1}, solid}, scaled y ticks={false}, ylabel={Error $\|\Sigma_{\mathcal{H}} - \widetilde{\Sigma}_{\mathcal{H}}\|_{\mathcal{H}_2}$}, y tick style={color={rgb,1:red,0.0;green,0.0;blue,0.0}, opacity={1.0}}, y tick label style={color={rgb,1:red,0.0;green,0.0;blue,0.0}, opacity={1.0}, rotate={0}}, ylabel style={at={(ticklabel cs:0.5)}, anchor=near ticklabel, at={{(ticklabel cs:0.5)}}, anchor={near ticklabel}, font={{\fontsize{11 pt}{14.3 pt}\selectfont}}, color={rgb,1:red,0.0;green,0.0;blue,0.0}, draw opacity={1.0}, rotate={0.0}}, ymode={log}, log basis y={10}, ymajorgrids={true}, ymin={0.01}, ymax={1.0}, yticklabels={{$10^{-2}$,$10^{-1}$,$10^{0}$}}, ytick={{0.01,0.1,1.0}}, ytick align={inside}, yticklabel style={font={{\fontsize{8 pt}{10.4 pt}\selectfont}}, color={rgb,1:red,0.0;green,0.0;blue,0.0}, draw opacity={1.0}, rotate={0.0}}, y grid style={color={rgb,1:red,0.0;green,0.0;blue,0.0}, draw opacity={0.1}, line width={0.5}, solid}, axis y line*={left}, y axis line style={color={rgb,1:red,0.0;green,0.0;blue,0.0}, draw opacity={1.0}, line width={1}, solid}, colorbar={false}]
    \addplot[color={rgb,1:red,0.0078;green,0.2431;blue,1.0}, name path={9}, draw opacity={1.0}, line width={1}, solid, mark={square*}, mark size={3.0 pt}, mark repeat={1}, mark options={color={rgb,1:red,0.0078;green,0.2431;blue,1.0}, draw opacity={1.0}, fill={rgb,1:red,0.0078;green,0.2431;blue,1.0}, fill opacity={1.0}, line width={0.75}, rotate={0}, solid}]
        table[row sep={\\}]
        {
            \\
            2.0  0.2814830895424734  \\
            4.0  0.19804136912393025  \\
            6.0  0.16870487825158606  \\
            8.0  0.12048999120934531  \\
            10.0  0.09007159831529021  \\
            12.0  0.06710278820323673  \\
            14.0  0.050345716483307386  \\
            16.0  0.037246372103305  \\
            18.0  0.025164418605815864  \\
            20.0  0.017015178284285174  \\
        }
        ;
    \addplot[color={rgb,1:red,0.102;green,0.7882;blue,0.2196}, name path={10}, draw opacity={1.0}, line width={1}, solid, mark={*}, mark size={3.0 pt}, mark repeat={1}, mark options={color={rgb,1:red,0.102;green,0.7882;blue,0.2196}, draw opacity={1.0}, fill={rgb,1:red,0.102;green,0.7882;blue,0.2196}, fill opacity={1.0}, line width={0.75}, rotate={0}, solid}]
        table[row sep={\\}]
        {
            \\
            2.0  0.39940135985940317  \\
            4.0  0.36925641458363295  \\
            6.0  0.20012143282751427  \\
            8.0  0.207980233918737  \\
            10.0  0.20811422276140318  \\
            12.0  0.20814929163796045  \\
            14.0  0.2081506097891487  \\
            16.0  0.20815067978267948  \\
            18.0  0.20815066927014045  \\
            20.0  0.2081506743535175  \\
        }
        ;
    \addplot[color={rgb,1:red,0.5451;green,0.1686;blue,0.8863}, name path={11}, draw opacity={1.0}, line width={1}, dashed, mark={star}, mark size={3.0 pt}, mark repeat={1}, mark options={color={rgb,1:red,0.5451;green,0.1686;blue,0.8863}, draw opacity={1.0}, fill={rgb,1:red,0.5451;green,0.1686;blue,0.8863}, fill opacity={1.0}, line width={0.75}, rotate={0}, solid}]
        table[row sep={\\}]
        {
            \\
            2.0  0.3994013598594074  \\
            4.0  0.36925641220968564  \\
            6.0  0.17961798981931315  \\
            8.0  0.1704688861320551  \\
            10.0  0.15713013571911008  \\
            12.0  0.15591526806401296  \\
            14.0  0.16501131023782983  \\
            16.0  0.15564262174784235  \\
            18.0  0.16312896857114698  \\
            20.0  0.16326654439588795  \\
        }
        ;
    \addplot[color={rgb,1:red,0.9098;green,0.0;blue,0.0431}, name path={12}, draw opacity={1.0}, line width={1}, dashed, mark={diamond*}, mark size={3.0 pt}, mark repeat={1}, mark options={color={rgb,1:red,0.9098;green,0.0;blue,0.0431}, draw opacity={1.0}, fill={rgb,1:red,0.9098;green,0.0;blue,0.0431}, fill opacity={1.0}, line width={0.75}, rotate={0}, solid}]
        table[row sep={\\}]
        {
            \\
            2.0  0.39940038566798564  \\
            4.0  0.36924919762545255  \\
            6.0  0.17335279310238844  \\
            8.0  0.11189652173548899  \\
            10.0  0.07652470367592445  \\
            12.0  0.06317260217070061  \\
            14.0  0.05184885042010864  \\
            16.0  0.03896393192170189  \\
            18.0  0.035821981019610954  \\
            20.0  0.03397543969318774  \\
        }
        ;
\end{axis}
\end{tikzpicture}
        \\
        (a) Input-output dynamic $\Htwo$-error
        &
        (b) Hamiltonian dynamic $\Htwo$-error
        \\
    \end{tabular}
  }
    \caption{$\mathcal{H}_2$-error of the input-output dynamic and the Hamiltonian dynamic over the reduced orders in the mass-spring-damper example.}
    \label{fig:msd-h2}
\end{figure}

Applying energy matching to \pHIRKA results in a maximum improvement of only 0.78\,\% in the Hamiltonian dynamic $\mathcal{H}_2$-error across all reduced orders. Consequently, we have not included the energy matched \pHIRKA in \Cref{fig:msd-h2}, concluding that the \ROMs produced by \pHIRKA are already nearly optimal for this example. In contrast, applying energy matching to \PRBT yields a significant improvement, such as 81.28\,\% for $\stateDimRed = 16$ in the Hamiltonian dynamic $\mathcal{H}_2$-error.

Additionally, we show the error trajectories $\|\outVar(t)-\reduce{\outVar}(t)\|_2$ and $|\outVar_\calH(t)-\reduce{\outVar}_\calH(t)|$ in \Cref{fig:msd-output-error} for the \ROM with reduced order $\stateDimRed = 16$.
As input signal, we choose $\inpVar(t) = [\sin(t),\cos(t)]^\T$ and plot the trajectories for times $t>50$ at which the system response has approximately settled. These trajectories are in line with our observations from \Cref{fig:msd-h2}.
The output error of \pHIRKA \ROM is worse than the error of \PRBT (and \EMPRBT) for all $t>50$ by more than two orders of magnitude. As anticipated, the output errors of \PRBT remain unchanged before and after optimization. However, while the Hamiltonian error is initially the largest for \PRBT (before energy matching), it achieves results similar to \pHIRKA after applying our method.

\begin{figure}[htpb]
    \centering
    \ref{leg:msd-output-error}
    \begin{tabular}{cc}
        \input{plots/2_msd_output_error.tikz}
        &
        \input{plots/2_msd_hamiltonian_error.tikz}
        \\
        (a) Output error
        &
        (b) Hamiltonian error
    \end{tabular}
    \caption{Error trajectory of the output and the Hamiltonian.}\label{fig:msd-output-error}
\end{figure}

\subsection{Mass-spring-damper with \texorpdfstring{$X_{\mathrm{min}}$}{Xmin} as Hamiltonian}
\label{subsec:MSDxMINHam}
In this experiment, we investigate the findings of \Cref{sec:specialcase-prbt}, i.e., we analyze the situation when the \FOM Hessian of the Hamiltonian is given by the minimal solution of the \KYP inequality (which corresponds to the optimal choice for \pHIRKA \cite{BreU22}). In particular, we consider the mass-spring-damper system from the previous subsection and modify the Hamiltonian in the \FOM to be the minimal solution of the \KYP inequality~\eqref{eq:KYP} and transform the other matrices accordingly, see~\Cref{subsec:pH}. The $\Htwo$-error of \PRBT before and after optimization is presented in~\Cref{tab:msdXmmin-hamiltonian-h2}. We conclude that for this example, \PRBT already provides a close-to-optimal approximation of the Hamiltonian since the error is almost identical before and after the optimization.

\begin{table}[htb]
    \caption{Hamiltonian dynamic $\Htwo$-errors of \PRBT and \EMPRBT for the mass-spring-damper example with $X_{\mathrm{min}}$ as Hamiltonian.}
    \label{tab:msdXmmin-hamiltonian-h2}
    \begin{tabular}{cccccc}
        \toprule
        $r$ & 4 & 8 & 12 & 16 & 20\\
        \midrule
        $\textsf{PRBT}$ & $\num{4.11e-01}$ & $\num{1.02e-02}$ & $\num{3.88e-04}$ & $\num{3.62e-05}$ & $\num{2.64e-05}$\\
        $\textsf{EM-PRBT}$ & $\num{4.11e-01}$ & $\num{1.02e-02}$ & $\num{3.87e-04}$ & $\num{3.14e-05}$ & $\num{2.10e-05}$\\
        \bottomrule
        \end{tabular}      
\end{table}

\subsection{Linear poroelasticity}\label{sec:num-exp:poro}
In our third example, we apply our proposed method to Biot’s consolidation model for poroelasticity. A general \pH formulation was derived in~\cite{AltMU21}, and the system is also part of the \pH benchmark collection. The state-space dimension is $\stateDim=980$ with one input and one output. 

Without executing~\ref{wf:minreal}, the example shows numerical issues when solving the positive-real \AREs. In particular, the numerical solver obtains indefinite Gramians, which must be projected onto the set of positive definite matrices. After the projection, the relative Frobenius norm of the residuals of the equations are \num{3.14e-7} and \num{3.55e-8}. One reason for this behavior is that the system is numerically not controllable. 
However, after applying our Kalman-like decomposition, we obtain a minimal realization of order $\stateDimRed=132$, which introduces an $\mathcal{H}_2$-error of the input-output dynamic of $\num{7.04e-11}$ and the Hamiltonian dynamic of $\num{1.61e-11}$. 
Now, for the minimal realization, the relative Frobenius norm of the residual of both \AREs is of order $\num{1e-15}$. 
We conclude that applying the Kalman-like decomposition greatly improves the numerical stability of computing the positive-real Gramians.

In \Cref{fig:poro-h2}, we can observe that \pHIRKA leads to the better input-output dynamic $\Htwo$-error and also to the best Hamiltonian dynamic $\Htwo$-error at the same time. The \ROMs from \PRBT have a similar input-output dynamic $\Htwo$-error as \pHIRKA, but a significantly worse $\Htwo$-error for the Hamiltonian dynamic. We again apply our energy matching algorithm to improve the Hamiltonian dynamic of the \PRBT \ROMs. 
In this example, we compare the barrier-method \Cref{alg:enermatch} with the state-of-the-art open source \SDP solvers 
Hypatia~\cite{CoyKV22}, COSMO~\cite{GarCG21} and the commercial solver MOSEK~\cite{mosek} denoted with \texttt{EM-PRBT-Hypatia}, \texttt{EM-PRBT-COSMO}, and \texttt{EM-PRBT-MOSEK}, respectively. 
We observe that the barrier method provides the best results among the methods, especially for larger reduced orders.
Our method again significantly improves the error of the Hamiltonian dynamic of the \PRBT \ROMs (more than one order of magnitude). Nevertheless, in this example \pHIRKA provides the best \ROMs for both objectives.

\begin{figure}[htpb]
  \centering
  \ref{leg:poro-h2}
  \resizebox{0.99\textwidth}{!}{
  \begin{tabular}{cc}

\begin{tikzpicture}[/tikz/background rectangle/.style={fill={rgb,1:red,1.0;green,1.0;blue,1.0}, fill opacity={1.0}, draw opacity={1.0}}, show background rectangle]
\begin{axis}[point meta max={nan}, point meta min={nan}, legend cell align={left}, legend columns={3}, title={}, title style={at={{(0.5,1)}}, anchor={south}, font={{\fontsize{14 pt}{18.2 pt}\selectfont}}, color={rgb,1:red,0.0;green,0.0;blue,0.0}, draw opacity={1.0}, rotate={0.0}, align={center}}, legend style={color={rgb,1:red,0.0;green,0.0;blue,0.0}, draw opacity={1.0}, line width={1}, solid, fill={rgb,1:red,1.0;green,1.0;blue,1.0}, fill opacity={1.0}, text opacity={1.0}, font={{\fontsize{8 pt}{10.4 pt}\selectfont}}, text={rgb,1:red,0.0;green,0.0;blue,0.0}, cells={anchor={center}}, at={(0.02, 0.02)}, anchor={south west}}, axis background/.style={fill={rgb,1:red,1.0;green,1.0;blue,1.0}, opacity={1.0}}, anchor={north west}, xshift={1.0mm}, yshift={-1.0mm}, width={0.43\textwidth}, height={0.2866666666666667\textwidth}, scaled x ticks={false}, xlabel={Reduced order $r$}, x tick style={color={rgb,1:red,0.0;green,0.0;blue,0.0}, opacity={1.0}}, x tick label style={color={rgb,1:red,0.0;green,0.0;blue,0.0}, opacity={1.0}, rotate={0}}, xlabel style={at={(ticklabel cs:0.5)}, anchor=near ticklabel, at={{(ticklabel cs:0.5)}}, anchor={near ticklabel}, font={{\fontsize{11 pt}{14.3 pt}\selectfont}}, color={rgb,1:red,0.0;green,0.0;blue,0.0}, draw opacity={1.0}, rotate={0.0}}, xmajorgrids={true}, xmin={2}, xmax={20}, xticklabels={{$2$,$4$,$6$,$8$,$10$,$12$,$14$,$16$,$18$,$20$}}, xtick={{2.0,4.0,6.0,8.0,10.0,12.0,14.0,16.0,18.0,20.0}}, xtick align={inside}, xticklabel style={font={{\fontsize{8 pt}{10.4 pt}\selectfont}}, color={rgb,1:red,0.0;green,0.0;blue,0.0}, draw opacity={1.0}, rotate={0.0}}, x grid style={color={rgb,1:red,0.0;green,0.0;blue,0.0}, draw opacity={0.1}, line width={0.5}, solid}, axis x line*={left}, x axis line style={color={rgb,1:red,0.0;green,0.0;blue,0.0}, draw opacity={1.0}, line width={1}, solid}, scaled y ticks={false}, ylabel={Error $\|\Sigma_{\mathrm{pH}} - \widetilde{\Sigma}_{\mathrm{pH}}\|_{\mathcal{H}_2}$}, y tick style={color={rgb,1:red,0.0;green,0.0;blue,0.0}, opacity={1.0}}, y tick label style={color={rgb,1:red,0.0;green,0.0;blue,0.0}, opacity={1.0}, rotate={0}}, ylabel style={at={(ticklabel cs:0.5)}, anchor=near ticklabel, at={{(ticklabel cs:0.5)}}, anchor={near ticklabel}, font={{\fontsize{11 pt}{14.3 pt}\selectfont}}, color={rgb,1:red,0.0;green,0.0;blue,0.0}, draw opacity={1.0}, rotate={0.0}}, ymode={log}, log basis y={10}, ymajorgrids={true}, ymin={1.0e-6}, ymax={0.001}, yticklabels={{$10^{-6}$,$10^{-5}$,$10^{-4}$,$10^{-3}$}}, ytick={{1.0e-6,1.0e-5,0.0001,0.001}}, ytick align={inside}, yticklabel style={font={{\fontsize{8 pt}{10.4 pt}\selectfont}}, color={rgb,1:red,0.0;green,0.0;blue,0.0}, draw opacity={1.0}, rotate={0.0}}, y grid style={color={rgb,1:red,0.0;green,0.0;blue,0.0}, draw opacity={0.1}, line width={0.5}, solid}, axis y line*={left}, y axis line style={color={rgb,1:red,0.0;green,0.0;blue,0.0}, draw opacity={1.0}, line width={1}, solid}, colorbar={false}, legend to name={leg:poro-h2}]
    \addplot[color={rgb,1:red,0.0078;green,0.2431;blue,1.0}, name path={19}, draw opacity={1.0}, line width={1}, solid, mark={square*}, mark size={3.0 pt}, mark repeat={1}, mark options={color={rgb,1:red,0.0078;green,0.2431;blue,1.0}, draw opacity={1.0}, fill={rgb,1:red,0.0078;green,0.2431;blue,1.0}, fill opacity={1.0}, line width={0.75}, rotate={0}, solid}]
        table[row sep={\\}]
        {
            \\
            2.0  0.0007447361518742896  \\
            4.0  6.788772553685912e-5  \\
            6.0  9.372042193601958e-6  \\
            8.0  9.367264014360192e-6  \\
            10.0  5.473144350781884e-6  \\
            12.0  5.124539652389795e-6  \\
            14.0  4.8200033316768885e-6  \\
            16.0  2.495319544122903e-6  \\
            18.0  1.5081702577110406e-6  \\
            20.0  1.4213942866377913e-6  \\
        }
        ;
    \addlegendentry {$\textsf{pHIRKA}$}
    \addplot[color={rgb,1:red,0.102;green,0.7882;blue,0.2196}, name path={20}, draw opacity={1.0}, line width={1}, solid, mark={*}, mark size={3.0 pt}, mark repeat={1}, mark options={color={rgb,1:red,0.102;green,0.7882;blue,0.2196}, draw opacity={1.0}, fill={rgb,1:red,0.102;green,0.7882;blue,0.2196}, fill opacity={1.0}, line width={0.75}, rotate={0}, solid}]
        table[row sep={\\}]
        {
            \\
            2.0  0.00011769280211868609  \\
            4.0  0.00011307863201360767  \\
            6.0  3.9916360075554404e-5  \\
            8.0  2.4132379351257684e-5  \\
            10.0  1.0626799677344938e-5  \\
            12.0  1.1450219874670399e-5  \\
            14.0  5.101573845251127e-6  \\
            16.0  3.2947763733386213e-6  \\
            18.0  2.404052507504122e-6  \\
            20.0  1.9802042918903715e-6  \\
        }
        ;
    \addlegendentry {$\textsf{PRBT}$}
    \addplot[color={rgb,1:red,0.9098;green,0.0;blue,0.0431}, name path={21}, draw opacity={1.0}, line width={1}, dashed, mark={diamond*}, mark size={3.0 pt}, mark repeat={1}, mark options={color={rgb,1:red,0.9098;green,0.0;blue,0.0431}, draw opacity={1.0}, fill={rgb,1:red,0.9098;green,0.0;blue,0.0431}, fill opacity={1.0}, line width={0.75}, rotate={0}, solid}]
        table[row sep={\\}]
        {
            \\
            2.0  0.00011769280203588295  \\
            4.0  0.00011307863201368916  \\
            6.0  3.9916357778087154e-5  \\
            8.0  2.4132377470669195e-5  \\
            10.0  1.0626801497079002e-5  \\
            12.0  1.1450219687390157e-5  \\
            14.0  5.101573845293595e-6  \\
            16.0  3.2947764159606797e-6  \\
            18.0  2.4040526007557766e-6  \\
            20.0  1.9802042918983476e-6  \\
        }
        ;
    \addlegendentry {$\textsf{EM-PRBT}$}
    \addplot[color={rgb,1:red,1.0;green,0.4863;blue,0.0}, name path={22}, draw opacity={1.0}, line width={1}, dashed, mark={triangle*}, mark size={3.0 pt}, mark repeat={1}, mark options={color={rgb,1:red,1.0;green,0.4863;blue,0.0}, draw opacity={1.0}, fill={rgb,1:red,1.0;green,0.4863;blue,0.0}, fill opacity={1.0}, line width={0.75}, rotate={0}, solid}]
        table[row sep={\\}]
        {
            \\
            2.0  0.00011769280211868609  \\
            4.0  0.0001130786231195624  \\
            6.0  3.991635478089371e-5  \\
            8.0  2.413237743206182e-5  \\
            10.0  1.0626801497017457e-5  \\
            12.0  1.1450219874840249e-5  \\
            14.0  5.101573765503182e-6  \\
            16.0  3.2947764148333834e-6  \\
            18.0  2.404052600790092e-6  \\
            20.0  1.9802041266168e-6  \\
        }
        ;
    \addlegendentry {$\textsf{EM-PRBT-Hypatia}$}
    \addplot[color={rgb,1:red,0.9451;green,0.298;blue,0.7569}, name path={23}, draw opacity={1.0}, line width={1}, dashed, mark={triangle*}, mark size={3.0 pt}, mark repeat={1}, mark options={color={rgb,1:red,0.9451;green,0.298;blue,0.7569}, draw opacity={1.0}, fill={rgb,1:red,0.9451;green,0.298;blue,0.7569}, fill opacity={1.0}, line width={0.75}, rotate={180}, solid}]
        table[row sep={\\}]
        {
            \\
            2.0  0.0001176928021538501  \\
            4.0  0.00011307863201377456  \\
            6.0  3.991635842381721e-5  \\
            8.0  2.4132380231842115e-5  \\
            10.0  1.0626798237447744e-5  \\
            12.0  1.145021664460753e-5  \\
            14.0  5.101573845370798e-6  \\
            16.0  3.2947762820463284e-6  \\
            18.0  2.4040526007784516e-6  \\
            20.0  1.9802041233079375e-6  \\
        }
        ;
    \addlegendentry {$\textsf{EM-PRBT-COSMO}$}
    \addplot[color={rgb,1:red,0.6392;green,0.6392;blue,0.6392}, name path={24}, draw opacity={1.0}, line width={1}, dashed, mark={star}, mark size={3.0 pt}, mark repeat={1}, mark options={color={rgb,1:red,0.6392;green,0.6392;blue,0.6392}, draw opacity={1.0}, fill={rgb,1:red,0.6392;green,0.6392;blue,0.6392}, fill opacity={1.0}, line width={0.75}, rotate={0}, solid}]
        table[row sep={\\}]
        {
            \\
            2.0  0.00011769280211868609  \\
            4.0  0.0001130786320137393  \\
            6.0  3.9916363679581305e-5  \\
            8.0  2.413237749073703e-5  \\
            10.0  1.0626801497112638e-5  \\
            12.0  1.1450219895707953e-5  \\
            14.0  5.101573404979362e-6  \\
            16.0  3.294776505213787e-6  \\
            18.0  2.4040526007894207e-6  \\
            20.0  1.9802035620148367e-6  \\
        }
        ;
    \addlegendentry {$\textsf{EM-PRBT-MOSEK}$}
\end{axis}
\end{tikzpicture} 
    &

\begin{tikzpicture}[/tikz/background rectangle/.style={fill={rgb,1:red,1.0;green,1.0;blue,1.0}, fill opacity={1.0}, draw opacity={1.0}}, show background rectangle]
\begin{axis}[point meta max={nan}, point meta min={nan}, legend cell align={left}, legend columns={1}, title={}, title style={at={{(0.5,1)}}, anchor={south}, font={{\fontsize{14 pt}{18.2 pt}\selectfont}}, color={rgb,1:red,0.0;green,0.0;blue,0.0}, draw opacity={1.0}, rotate={0.0}, align={center}}, legend style={color={rgb,1:red,0.0;green,0.0;blue,0.0}, draw opacity={1.0}, line width={1}, solid, fill={rgb,1:red,1.0;green,1.0;blue,1.0}, fill opacity={1.0}, text opacity={1.0}, font={{\fontsize{8 pt}{10.4 pt}\selectfont}}, text={rgb,1:red,0.0;green,0.0;blue,0.0}, cells={anchor={center}}, at={(1.02, 1)}, anchor={north west}}, axis background/.style={fill={rgb,1:red,1.0;green,1.0;blue,1.0}, opacity={1.0}}, anchor={north west}, xshift={1.0mm}, yshift={-1.0mm}, width={0.43\textwidth}, height={0.2866666666666667\textwidth}, scaled x ticks={false}, xlabel={Reduced order $r$}, x tick style={color={rgb,1:red,0.0;green,0.0;blue,0.0}, opacity={1.0}}, x tick label style={color={rgb,1:red,0.0;green,0.0;blue,0.0}, opacity={1.0}, rotate={0}}, xlabel style={at={(ticklabel cs:0.5)}, anchor=near ticklabel, at={{(ticklabel cs:0.5)}}, anchor={near ticklabel}, font={{\fontsize{11 pt}{14.3 pt}\selectfont}}, color={rgb,1:red,0.0;green,0.0;blue,0.0}, draw opacity={1.0}, rotate={0.0}}, xmajorgrids={true}, xmin={2}, xmax={20}, xticklabels={{$2$,$4$,$6$,$8$,$10$,$12$,$14$,$16$,$18$,$20$}}, xtick={{2.0,4.0,6.0,8.0,10.0,12.0,14.0,16.0,18.0,20.0}}, xtick align={inside}, xticklabel style={font={{\fontsize{8 pt}{10.4 pt}\selectfont}}, color={rgb,1:red,0.0;green,0.0;blue,0.0}, draw opacity={1.0}, rotate={0.0}}, x grid style={color={rgb,1:red,0.0;green,0.0;blue,0.0}, draw opacity={0.1}, line width={0.5}, solid}, axis x line*={left}, x axis line style={color={rgb,1:red,0.0;green,0.0;blue,0.0}, draw opacity={1.0}, line width={1}, solid}, scaled y ticks={false}, ylabel={Error $\|\Sigma_{\mathcal{H}} - \widetilde{\Sigma}_{\mathcal{H}}\|_{\mathcal{H}_2}$}, y tick style={color={rgb,1:red,0.0;green,0.0;blue,0.0}, opacity={1.0}}, y tick label style={color={rgb,1:red,0.0;green,0.0;blue,0.0}, opacity={1.0}, rotate={0}}, ylabel style={at={(ticklabel cs:0.5)}, anchor=near ticklabel, at={{(ticklabel cs:0.5)}}, anchor={near ticklabel}, font={{\fontsize{11 pt}{14.3 pt}\selectfont}}, color={rgb,1:red,0.0;green,0.0;blue,0.0}, draw opacity={1.0}, rotate={0.0}}, ymode={log}, log basis y={10}, ymajorgrids={true}, ymin={1.0e-7}, ymax={0.001}, yticklabels={{$10^{-7}$,$10^{-6}$,$10^{-5}$,$10^{-4}$,$10^{-3}$}}, ytick={{1.0e-7,1.0e-6,1.0e-5,0.0001,0.001}}, ytick align={inside}, yticklabel style={font={{\fontsize{8 pt}{10.4 pt}\selectfont}}, color={rgb,1:red,0.0;green,0.0;blue,0.0}, draw opacity={1.0}, rotate={0.0}}, y grid style={color={rgb,1:red,0.0;green,0.0;blue,0.0}, draw opacity={0.1}, line width={0.5}, solid}, axis y line*={left}, y axis line style={color={rgb,1:red,0.0;green,0.0;blue,0.0}, draw opacity={1.0}, line width={1}, solid}, colorbar={false}]
    \addplot[color={rgb,1:red,0.0078;green,0.2431;blue,1.0}, name path={25}, draw opacity={1.0}, line width={1}, solid, mark={square*}, mark size={3.0 pt}, mark repeat={1}, mark options={color={rgb,1:red,0.0078;green,0.2431;blue,1.0}, draw opacity={1.0}, fill={rgb,1:red,0.0078;green,0.2431;blue,1.0}, fill opacity={1.0}, line width={0.75}, rotate={0}, solid}]
        table[row sep={\\}]
        {
            \\
            2.0  8.975688225205634e-5  \\
            4.0  7.935972011476676e-6  \\
            6.0  1.0506485961535167e-6  \\
            8.0  1.0472916663928246e-6  \\
            10.0  6.030199239538883e-7  \\
            12.0  5.625155586203396e-7  \\
            14.0  5.347622761661598e-7  \\
            16.0  2.737808245541022e-7  \\
            18.0  1.5590556292305108e-7  \\
            20.0  1.4757779997106738e-7  \\
        }
        ;
    \addplot[color={rgb,1:red,0.102;green,0.7882;blue,0.2196}, name path={26}, draw opacity={1.0}, line width={1}, solid, mark={*}, mark size={3.0 pt}, mark repeat={1}, mark options={color={rgb,1:red,0.102;green,0.7882;blue,0.2196}, draw opacity={1.0}, fill={rgb,1:red,0.102;green,0.7882;blue,0.2196}, fill opacity={1.0}, line width={0.75}, rotate={0}, solid}]
        table[row sep={\\}]
        {
            \\
            2.0  2.9612571256503524e-5  \\
            4.0  0.00011801560071289715  \\
            6.0  2.0435158316556245e-5  \\
            8.0  1.533907632379485e-5  \\
            10.0  1.3997648585069026e-5  \\
            12.0  1.3518949130662594e-5  \\
            14.0  1.364709077954618e-5  \\
            16.0  1.3925465098268729e-5  \\
            18.0  1.3809893878502018e-5  \\
            20.0  1.3749587456808718e-5  \\
        }
        ;
    \addplot[color={rgb,1:red,0.9098;green,0.0;blue,0.0431}, name path={27}, draw opacity={1.0}, line width={1}, dashed, mark={diamond*}, mark size={3.0 pt}, mark repeat={1}, mark options={color={rgb,1:red,0.9098;green,0.0;blue,0.0431}, draw opacity={1.0}, fill={rgb,1:red,0.9098;green,0.0;blue,0.0431}, fill opacity={1.0}, line width={0.75}, rotate={0}, solid}]
        table[row sep={\\}]
        {
            \\
            2.0  2.039071331262881e-5  \\
            4.0  0.00011641061263838423  \\
            6.0  1.503028805306258e-5  \\
            8.0  7.286786171744609e-6  \\
            10.0  2.832534953029276e-6  \\
            12.0  2.31570287023596e-6  \\
            14.0  8.473247004257568e-7  \\
            16.0  8.68156088639726e-7  \\
            18.0  9.485328742577182e-7  \\
            20.0  1.2813154938143035e-6  \\
        }
        ;
    \addplot[color={rgb,1:red,1.0;green,0.4863;blue,0.0}, name path={28}, draw opacity={1.0}, line width={1}, dashed, mark={triangle*}, mark size={3.0 pt}, mark repeat={1}, mark options={color={rgb,1:red,1.0;green,0.4863;blue,0.0}, draw opacity={1.0}, fill={rgb,1:red,1.0;green,0.4863;blue,0.0}, fill opacity={1.0}, line width={0.75}, rotate={0}, solid}]
        table[row sep={\\}]
        {
            \\
            2.0  2.291207777867957e-5  \\
            4.0  0.0001185798241302823  \\
            6.0  2.5261403360735554e-5  \\
            8.0  1.3882869854179495e-5  \\
            10.0  6.841961724625612e-6  \\
            12.0  4.793343135288211e-6  \\
            14.0  2.6901039506759975e-6  \\
            16.0  1.701732694520087e-6  \\
            18.0  1.6701847659395938e-6  \\
            20.0  1.6626934796407806e-6  \\
        }
        ;
    \addplot[color={rgb,1:red,0.9451;green,0.298;blue,0.7569}, name path={29}, draw opacity={1.0}, line width={1}, dashed, mark={triangle*}, mark size={3.0 pt}, mark repeat={1}, mark options={color={rgb,1:red,0.9451;green,0.298;blue,0.7569}, draw opacity={1.0}, fill={rgb,1:red,0.9451;green,0.298;blue,0.7569}, fill opacity={1.0}, line width={0.75}, rotate={180}, solid}]
        table[row sep={\\}]
        {
            \\
            2.0  2.239434517656643e-5  \\
            4.0  2.317605159268264e-5  \\
            6.0  3.030283053870218e-5  \\
            8.0  2.2816325107513637e-5  \\
            10.0  1.9223109319393194e-5  \\
            12.0  1.0907517038134489e-5  \\
            14.0  9.631593735708395e-6  \\
            16.0  7.759960809116467e-6  \\
            18.0  8.49240261256201e-6  \\
            20.0  4.882693848470841e-6  \\
        }
        ;
    \addplot[color={rgb,1:red,0.6392;green,0.6392;blue,0.6392}, name path={30}, draw opacity={1.0}, line width={1}, dashed, mark={star}, mark size={3.0 pt}, mark repeat={1}, mark options={color={rgb,1:red,0.6392;green,0.6392;blue,0.6392}, draw opacity={1.0}, fill={rgb,1:red,0.6392;green,0.6392;blue,0.6392}, fill opacity={1.0}, line width={0.75}, rotate={0}, solid}]
        table[row sep={\\}]
        {
            \\
            2.0  2.2637702700513748e-5  \\
            4.0  0.0001536166910194745  \\
            6.0  2.6333531364293586e-5  \\
            8.0  1.3595496804721296e-5  \\
            10.0  7.285610793952706e-6  \\
            12.0  4.47560842805323e-6  \\
            14.0  2.603436375290329e-6  \\
            16.0  1.5677614677570473e-6  \\
            18.0  2.6679652133725573e-6  \\
            20.0  2.5796008512595593e-6  \\
        }
        ;
\end{axis}
\end{tikzpicture}
    \\
    (a) Input-output dynamic $\Htwo$-error
    &
    (b) Hamiltonian dynamic $\Htwo$-error
    \\
  \end{tabular}
}
  \caption{$\mathcal{H}_2$-error of the input-output dynamic and the Hamiltonian dynamic over the reduced orders in the poroelasticity example.}
  \label{fig:poro-h2}
\end{figure}
\section{Conclusions}
\label{sec:conclusion}
We introduced the view of \pH systems as an extended dynamical system in~\eqref{eq:ph:ext}, combining the standard input-output dynamic and the Hamiltonian dynamic. We studied how this view affects observability and have derived a corresponding structure-preserving Kalman-like decomposition in \Cref{thm:KalmanFormPHQO}, which can be used as a preprocessing tool for numerical methods. 
Using the observation that the \KYP inequality determines all possible Hamiltonians, we proposed a \MOR post-processing method called energy matching: Given a passive \ROM for the input-output dynamic, solely consider the optimization problem~\eqref{eq:energyErrMin} for finding the best Hamiltonian. We showed that this optimization problem is uniquely solvable and convex (see \Cref{thm:energyOptimizationUniqueSolution}). We presented two numerical approaches to solve this problem and demonstrated their feasibility using three academic examples. 
Subsequent research will involve a deeper exploration of the \MOR problem within the extended norm~\eqref{eq:extendedNorm}, considering all system matrices.

\subsection*{Acknowledgments}
We thank Prof.~Carsten Scherer (U Stuttgart) for valuable comments regarding \LMI problems.
We would like to express our sincere gratitude to the referees for their detailed and insightful comments. Their valuable feedback significantly improved the quality and clarity of this paper.
P.~Schwerdtner acknowledges funding from the DFG within the project 424221635.
T.~Holicki, J.~Nicodemus and B.~Unger acknowledge funding from the DFG under Germany's Excellence Strategy -- EXC 2075 -- 390740016 and are thankful for support by the Stuttgart Center for Simulation Science (SimTech). 

\bibliographystyle{plain-doi}
\bibliography{journalabbr, literature}   

\begin{thebibliography}{10}

\bibitem{AfkH17}
\textsc{B.~M. Afkham and J.~S. Hesthaven}.
\newblock \href{https://doi.org/10.1137/17M1111991}{Structure preserving model
  reduction of parametric {Hamiltonian} systems}.
\newblock {\em {SIAM} J. Sci. Comput.}, 39(6):A2616--A2644, 2017.

\bibitem{AfkH19}
\textsc{B.~M. Afkham and J.~S. Hesthaven}.
\newblock \href{https://doi.org/10.1007/s10915-018-0653-6}{Structure-preserving
  model-reduction of dissipative {Hamiltonian} systems}.
\newblock {\em J. Sci. Comput.}, 81:3--21, 2019.

\bibitem{AltMU21}
\textsc{R.~{Altmann}, V.~{Mehrmann}, and B.~{Unger}}.
\newblock
  \href{https://doi.org/10.1080/13873954.2021.1975137}{Port-{Hamiltonian}
  formulations of poroelastic network models}.
\newblock {\em Math. Comput. Model. Dyn. Sys.}, 27(1):429--452, 2021.

\bibitem{AltS17}
\textsc{R.~Altmann and P.~Schulze}.
\newblock \href{https://doi.org/10.1016/j.sysconle.2016.12.005}{A
  port-{Hamiltonian} formulation of the {Navier}--{Stokes} equations for
  reactive flows}.
\newblock {\em Syst. Control Lett.}, 100:51--55, 2017.

\bibitem{Ant05}
\textsc{A.~C. Antoulas}.
\newblock \href{https://doi.org/10.1137/1.9780898718713}{{\em Approximation of
  large-scale dynamical systems}}.
\newblock Advances in Design and Control. SIAM, Philadelphia, PA, USA, 2005.

\bibitem{AntBG20}
\textsc{A.~C. Antoulas, C.~Beattie, and S.~Gugercin}.
\newblock \href{https://doi.org/10.1137/1.9781611976083}{{\em Interpolatory
  Methods for Model Reduction}}.
\newblock Computational Science \& Engineering. SIAM, Philadelphia, PA, USA,
  2020.

\bibitem{mosek}
\textsc{M.~ApS}.
\newblock \href{https://docs.mosek.com/10.1/juliaapi/index.html}{{\em MOSEK
  Optimizer API for Julia. Version 10.1.}}, 2019.

\bibitem{CarFFHT22}
\textsc{F.~Bagge~Carlson, M.~F{\"a}lt, A.~Heimerson, and O.~Troeng}.
\newblock
  \href{https://doi.org/10.1109/CDC45484.2021.9683403}{{ControlSystems.jl}: A
  control toolbox in {Julia}}.
\newblock In {\em Proc. 60th IEEE Conference on Decision and Control (CDC)
  2021, Austin, TX, USA}, pages 4847--4853, 2022.

\bibitem{BeaMX22}
\textsc{C.~Beattie, V.~Mehrmann, and H.~Xu}.
\newblock \href{https://doi.org/10.48550/arXiv.2201.05355}{Port-{Hamiltonian}
  realizations of linear time invariant systems}.
\newblock {\em ArXiv e-print 2201.05355}, 2022.

\bibitem{BeaMXZ18}
\textsc{C.~Beattie, V.~Mehrmann, H.~Xu, and H.~Zwart}.
\newblock \href{https://doi.org/10.1007/s00498-018-0223-3}{Port-{Hamiltonian}
  descriptor systems}.
\newblock {\em Math. Control Signals Systems}, 30(17):1--27, 2018.

\bibitem{BenCOW17}
\textsc{P.~Benner, A.~Cohen, M.~Ohlberger, and K.~Willcox}.
\newblock \href{https://doi.org/10.1137/1.9781611974829}{{\em Model Reduction
  and Approximation}}.
\newblock Computational Science \& Engineering. SIAM, Philadelphia, PA, 2017.

\bibitem{BenGP22}
\textsc{P.~Benner, P.~Goyal, and I.~Pontes~Duff}.
\newblock \href{https://doi.org/10.1109/TAC.2021.3086319}{Gramians, energy
  functionals, and balanced truncation for linear dynamical systems with
  quadratic outputs}.
\newblock {\em {IEEE} Trans. Automat. Control}, 67(2), 2022.

\bibitem{BorSF21}
\textsc{P.~Borja, J.~M.~A. Scherpen, and K.~Fujimoto}.
\newblock \href{https://doi.org/10.1109/TAC.2021.3138645}{Extended balancing of
  continuous {LTI} systems: a structure-preserving approach}.
\newblock {\em {IEEE} Trans. Automat. Control}, 68(1):257--271, 2021.

\bibitem{BoyV04}
\textsc{S.~P. Boyd and L.~Vandenberghe}.
\newblock \href{https://doi.org/10.1017/CBO9780511804441}{{\em Convex
  optimization}}.
\newblock Cambridge University Press, 2004.

\bibitem{BreMS22}
\textsc{T.~{Breiten}, R.~{Morandin}, and P.~{Schulze}}.
\newblock \href{https://doi.org/10.1016/j.camwa.2021.07.022}{Error bounds for
  port-{Hamiltonian} model and controller reduction based on system balancing}.
\newblock {\em Comput. Math. Appl.}, 116:100--115, 2022.

\bibitem{BreU22}
\textsc{T.~{Breiten} and B.~{Unger}}.
\newblock \href{https://doi.org/10.1016/j.automatica.2022.110368}{Passivity
  preserving model reduction via spectral factorization}.
\newblock {\em Automatica J. IFAC}, 142:110368, 2022.

\bibitem{BucBH19}
\textsc{P.~Buchfink, A.~Bhatt, and B.~Haasdonk}.
\newblock \href{https://doi.org/10.3390/mca24020043}{Symplectic model order
  reduction with non-orthonormal bases}.
\newblock {\em Math. Comput. Appl.}, 24(2):43, 2019.

\bibitem{CalDRSK22}
\textsc{F.~Califano, A.~Dijkshoorn, S.~Roodink, S.~Stramigioli, and
  G.~Krijnen}.
\newblock \href{https://doi.org/10.1109/TMECH.2022.3192324}{Energy-aware
  control of {Euler}–{Bernoulli} beams by means of an axial load}.
\newblock {\em IEEE/ASME Trans. Mechatron.}, 27(6):5959--5968, 2022.

\bibitem{CamIV14}
\textsc{M.~K. Camlibel, L.~Iannelli, and F.~Vasca}.
\newblock \href{https://doi.org/10.1007/s10107-013-0678-4}{Passivity and
  complementarity}.
\newblock {\em Math. Program.}, 145:531--563, 2014.

\bibitem{CheGH23}
\textsc{K.~Cherifi, H.~Gernandt, and D.~Hinsen}.
\newblock \href{https://doi.org/10.1007/s00498-023-00373-2}{The difference
  between port-{Hamiltonian}, passive and positive real descriptor systems}.
\newblock {\em Math. Control Signals Systems}, 2023.

\bibitem{CoyKV22}
\textsc{C.~Coey, L.~Kapelevich, and J.~P. Vielma}.
\newblock \href{https://doi.org/10.1287/ijoc.2022.1202}{Solving natural conic
  formulations with {Hypatia.jl}}.
\newblock {\em INFORMS J. Comput.}, 34(5):2686--2699, 2022.

\bibitem{DesP84}
\textsc{U.~Desai and D.~Pal}.
\newblock \href{https://doi.org/10.1109/TAC.1984.1103438}{A transformation
  approach to stochastic model reduction}.
\newblock {\em {IEEE} Trans. Automat. Control}, 29(12):1097--1100, 1984.

\bibitem{EggKLMM18}
\textsc{H.~Egger, T.~Kugler, B.~Liljegren-Sailer, N.~Marheineke, and
  V.~Mehrmann}.
\newblock \href{https://doi.org/10.1137/17M1125303}{On structure preserving
  model reduction for damped wave propagation in transport networks}.
\newblock {\em {SIAM} J. Sci. Comput.}, 40:A331--A365, 2018.

\bibitem{EstT00}
\textsc{D.~Est{\'e}vez-Schwarz and C.~Tischendorf}.
\newblock
  \href{https://doi.org/10.1002/(SICI)1097-007X(200003/04)28:2%3C131::AID-CTA100%3E3.0.CO;2-W}{Structural
  analysis for electrical circuits and consequences for {MNA}}.
\newblock {\em Int J. Circ. Theor. Appl.}, 28:131--162, 2000.

\bibitem{GarCG21}
\textsc{M.~Garstka, M.~Cannon, and P.~Goulart}.
\newblock \href{https://doi.org/10.1007/s10957-021-01896-x}{{COSMO}: A conic
  operator splitting method for convex conic problems}.
\newblock {\em J. Optim. Theory Appl.}, 190(3):779--810, 2021.

\bibitem{GosA19}
\textsc{I.~V. Gosea and A.~C. Antoulas}.
\newblock \href{https://doi.org/10.1109/CDC40024.2019.9030025}{A two-sided
  iterative framework for model reduction of linear systems with quadratic
  output}.
\newblock In {\em Proc. 58th IEEE Conf. Decision Control (CDC) 2019, Nice,
  France}, pages 7812--7817, 2019.

\bibitem{GosG22}
\textsc{I.~V. Gosea and S.~Gugercin}.
\newblock \href{https://doi.org/10.1007/s10915-022-01771-5}{Data-driven
  modeling of lineare dynamical systems with quadratic output in the {AAA}
  framework}.
\newblock {\em J. Sci. Comput.}, 91(16), 2022.

\bibitem{GugPBV12}
\textsc{S.~Gugercin, R.~V. Polyuga, C.~Beattie, and A.~{van der Schaft}}.
\newblock
  \href{https://doi.org/10.1016/j.automatica.2012.05.052}{Structure-preserving
  tangential interpolation for model reduction of port-{Hamiltonian} systems}.
\newblock {\em Automatica J. IFAC}, 48(9):1963--1974, 2012.

\bibitem{GunF99a}
\textsc{M.~G{\"u}nther and U.~Feldmann}.
\newblock {CAD}-based electric-circuit modeling in industry. {I. M}athematical
  structure and index of network equations.
\newblock {\em Surv. Math. Ind.}, 8:97--129, 1999.

\bibitem{GunF99b}
\textsc{M.~G{\"u}nther and U.~Feldmann}.
\newblock {CAD}-based electric-circuit modeling in industry. {II. I}mpact of
  circuit configurations and parameters.
\newblock {\em Surv. Math. Ind.}, 8:131--157, 1999.

\bibitem{LanT85}
\textsc{P.~Lancaster and M.~Tismenetsky}.
\newblock
  \href{https://www.elsevier.com/books/the-theory-of-matrices/lancaster/978-0-08-051908-1}{{\em
  The Theory of Matrices: With Applications}}.
\newblock Computer Science and Scientific Computing. Elsevier Science, 2
  edition, 1985.

\bibitem{LubDGHBV23}
\textsc{M.~Lubin, O.~Dowson, J.~{Dias Garcia}, J.~Huchette, B.~Legat, and J.~P.
  Vielma}.
\newblock \href{https://doi.org/10.1007/s12532-023-00239-3}{{JuMP} 1.0:
  {Recent} improvements to a modeling language for mathematical optimization}.
\newblock {\em Math. Program. Comput.}, 15:581–589, 2023.

\bibitem{Mac13}
\textsc{U.~Mackenroth}.
\newblock \href{https://doi.org/10.1007/978-3-662-09775-5}{{\em Robust Control
  Systems: Theory and Case Studies}}.
\newblock Springer Science \& Business Media, 2013.

\bibitem{MagN19}
\textsc{J.~R. Magnus and H.~Neudecker}.
\newblock \href{https://doi.org/10.1002/9781119541219}{{\em Matrix Differential
  Calculus with Applications in Statistics and Econometrics}}.
\newblock John Wiley \& Sons, 2019.

\bibitem{MehU22}
\textsc{V.~Mehrmann and B.~Unger}.
\newblock \href{https://doi.org/10.1017/S0962492922000083}{Control of
  port-{Hamiltonian} differential-algebraic systems and applications}.
\newblock {\em Acta Numer.}, 32:395–515, 2023.

\bibitem{MogR18}
\textsc{P.~K. Mogensen and A.~N. Riseth}.
\newblock \href{https://doi.org/10.21105/joss.00615}{Optim: A mathematical
  optimization package for {Julia}}.
\newblock {\em J. Open Source Softw.}, 3(24):615--618, 2018.

\bibitem{MorNU23}
\textsc{R.~Morandin, J.~Nicodemus, and B.~Unger}.
\newblock \href{https://doi.org/10.1137/22M149329X}{Port-{Hamiltonian}
  {Dynamic} {Mode} {Decomposition}}.
\newblock {\em {SIAM} J. Sci. Comput.}, 45(4):A1690--A1710, 2023.

\bibitem{MosL20}
\textsc{T.~Moser and B.~Lohmann}.
\newblock \href{https://doi.org/10.1109/CDC42340.2020.9304134}{A new
  {R}iemannian framework for efficient $\mathcal{H}_2$-optimal model reduction
  of port-{Hamiltonian} systems}.
\newblock In {\em Proc. 59th IEEE Conf. on Decision and Control (CDC)}, pages
  5043--5049, 2020.

\bibitem{Obe91}
\textsc{R.~Ober}.
\newblock \href{https://doi.org/10.1137/0329065}{Balanced {Parametrization} of
  {Classes} of {Linear} {Systems}}.
\newblock {\em {SIAM} J. Cont. Optim.}, 29(6):1251--1287, 1991.

\bibitem{PenM16}
\textsc{L.~Peng and K.~Mohseni}.
\newblock \href{https://doi.org/10.1137/140978922}{Symplectic model reduction
  of {Hamiltonian} systems}.
\newblock {\em {SIAM} J. Sci. Comput.}, 38(1):A1--A27, 2016.

\bibitem{PolS08}
\textsc{R.~V. Polyuga and A.~van~der Schaft}.
\newblock \href{https://scholar.lib.vt.edu/MTNS/Papers/175.pdf}{Structure
  preserving model reduction of port-{Hamiltonian} systems}.
\newblock In {\em Proceedings of the 18th International Symposium on
  Mathematical Theory of Networks and Systems (MTNS 2008), Blacksburg,
  Virginia, USA, July 28 - August 1}, 2008.

\bibitem{PolS10}
\textsc{R.~V. Polyuga and A.~{van der Schaft}}.
\newblock \href{https://doi.org/10.1016/j.automatica.2010.01.018}{Structure
  preserving model reduction of port-{{Hamiltonian}} systems by moment matching
  at infinity}.
\newblock {\em Automatica J. IFAC}, 46(4):665--672, 2010.

\bibitem{PrzPGB24}
\textsc{J.~Przybilla, I.~P. Duff, P.~Goyal, and P.~Benner}.
\newblock \href{https://doi.org/10.48550/arXiv.2402.14716}{Balanced truncation
  of descriptor systems with a quadratic output}.
\newblock {\em ArXiv e-print 2402.14716}, 2024.

\bibitem{PulN19}
\textsc{R.~Pulch and A.~Narayan}.
\newblock \href{https://doi.org/10.1137/17M1148797}{Balanced truncation for
  model order reduction of linear dynamical systems with quadratic outputs}.
\newblock {\em {SIAM} J. Sci. Comput.}, 41(4):A2270--A2295, 2019.

\bibitem{RasBZSJF22}
\textsc{R.~Rashad, D.~Bicego, J.~Zult, S.~Sanchez-Escalonilla, R.~Jiao,
  A.~Franchi, and S.~Stramigioli}.
\newblock \href{https://doi.org/10.1109/TRO.2022.3183532}{Energy aware
  impedance control of a flying end-effector in the port-{Hamiltonian}
  framework}.
\newblock {\em IEEE Trans. Robot.}, 38(6):3936--3955, 2022.

\bibitem{ReiRV15}
\textsc{T.~Reis, O.~Rendel, and M.~Voigt}.
\newblock \href{https://doi.org/10.1016/j.laa.2015.06.021}{The
  {K}alman--{Y}akubovich--{P}opov inequality for differential-algebraic
  systems}.
\newblock {\em Linear Algebra Appl.}, 485:153--193, 2015.

\bibitem{ReiS10}
\textsc{T.~Reis and T.~Stykel}.
\newblock \href{https://doi.org/10.1080/00207170903100214}{Positive real and
  bounded real balancing for model reduction of descriptor systems}.
\newblock {\em Internat. J. Control}, 83(1):74--88, 2010.

\bibitem{SatS18}
\textsc{K.~Sato and H.~Sato}.
\newblock \href{https://doi.org/10.1109/TAC.2017.2723259}{{Structure-Preserving
  $H^2$ Optimal Model Reduction Based on the Riemannian Trust-Region Method}}.
\newblock {\em {IEEE} Trans. Automat. Control}, 63(2):505--512, 2018.

\bibitem{SchV20}
\textsc{P.~Schwerdtner and M.~Voigt}.
\newblock \href{https://doi.org/10.1137/20M1380235}{{SOBMOR}: Structured
  optimization-based model order reduction}.
\newblock {\em {SIAM} J. Sci. Comput.}, 45(2):A502--A529, 2023.

\bibitem{Stu99}
\textsc{J.~F. Sturm}.
\newblock \href{https://doi.org/10.1080/10556789908805766}{Using {SeDuMi} 1.02,
  a {MATLAB} toolbox for optimization over symmetric cones}.
\newblock {\em Optim. Methods Softw.}, 11(1-4):625--653, 1999.

\bibitem{VanVLM12}
\textsc{R.~Van~Beeumen, K.~Van~Nimmen, G.~Lombaert, and K.~Meerbergen}.
\newblock \href{https://doi.org/10.1002/nme.4255}{Model reduction for dynamical
  systems with quadratic output}.
\newblock {\em Internat. J. Numer. Methods Eng.}, 91:229--248, 2012.

\bibitem{VanM10}
\textsc{R.~Van~Beeumen and K.~Meerbergen}.
\newblock \href{https://doi.org/10.1063/1.3498345}{Model reduction by balanced
  truncation of linear systems with a quadratic output}.
\newblock {\em AIP Conference Proceedings}, 1281(1):2033--2036, 2010.

\bibitem{SchJ14}
\textsc{A.~{van der Schaft} and D.~Jeltsema}.
\newblock \href{https://doi.org/10.1561/2600000002}{Port-{Hamiltonian} systems
  theory: {An} introductory overview}.
\newblock {\em Foundations and Trends in Systems and Control}, 1(2-3):173--378,
  2014.

\bibitem{Wil71}
\textsc{J.~Willems}.
\newblock \href{https://doi.org/10.1109/TAC.1971.1099831}{Least squares
  stationary optimal control and the algebraic {R}iccati equation}.
\newblock {\em {IEEE} Trans. Automat. Control}, 16(6):621--634, 1971.

\bibitem{Wil72}
\textsc{J.~C. Willems}.
\newblock \href{https://doi.org/10.1007/BF00276494}{Dissipative dynamical
  systems part ii: Linear systems with quadratic supply rates}.
\newblock {\em Arch. Ration. Mech. Anal.}, 45:352--393, 1972.

\end{thebibliography}

\appendix
\section{Proof of \texorpdfstring{\Cref{prop:gradientBarrierFunction}}{Proposition 5.7}}
\label{app:proof-prop:gradientBarrierFunction}
Using the chain rule, we obtain
  	\begin{align*}
        \nabla_X 
    \ln\left(\det\left(\calW_{\systemPHRed}(X)\right)\right)
        &= \frac{1}{\det\left(\calW_{\systemPHRed}(X)\right)} 
        \det\left(\calW_{\systemPHRed}(X)\right)\tr\left(\inv{\left(\calW_{\systemPHRed}(X)\right)} 
        \nabla_X \left(\calW_{\systemPHRed}(X)\right)\right).
    \end{align*}
    The directional derivative of $\calW_{\systemPHRed}(X)$ is given by
    \begin{equation*}
        \begin{aligned}
            \calD_{\Delta_X} \calW_{\systemPHRed}(X)&= 
            \begin{bmatrix}
                -\reduce{A}^\T \Delta_X - \Delta_X \reduce{A} 
                & -\Delta_X \reduce{B}\\
                -\reduce{B}^\T \Delta_X & 0
            \end{bmatrix}
            = 
            \begin{bmatrix}
                -\reduce{A}^\T \\
                -\reduce{B}^\T
            \end{bmatrix}
            \Delta_X
            \begin{bmatrix}
                I & 0
            \end{bmatrix} 
            + 
            \begin{bmatrix}
                I \\ 
                0
            \end{bmatrix}
            \Delta_X
            \begin{bmatrix}
                -\reduce{A} & -\reduce{B}
            \end{bmatrix}
        \end{aligned}
    \end{equation*}
    resulting in 
    \begin{multline*}
          	\calD_{\Delta_X} \ln\left(\det\left(\calW_{\systemPHRed}(X)\right)\right)\\
            =\tr\left(\begin{bmatrix}
                I & 0
              \end{bmatrix}
              \inv{\left(\calW_{\systemPHRed}(X)\right)}
              \begin{bmatrix}
                -\reduce{A}^\T \\
                -\reduce{B}^\T
              \end{bmatrix} \Delta_X
            \right) + 
            \tr\left(
              \begin{bmatrix}
                -\reduce{A} & -\reduce{B}
              \end{bmatrix}
              \inv{\left(\calW_{\systemPHRed}(X)\right)}
              \begin{bmatrix}
                I \\ 
                0
              \end{bmatrix}
              \Delta_X
            \right).\qedhere
    \end{multline*}
\end{document}